\documentclass{amsart}
\usepackage{amsfonts}
\usepackage{amsthm}
\usepackage{amsmath, pb-diagram}
\usepackage{amscd}
\usepackage[latin2]{inputenc}
\usepackage{t1enc}
\usepackage[mathscr]{eucal}
\usepackage{indentfirst}
\usepackage{graphicx}
\usepackage{graphics,epsfig}
\usepackage{pict2e}
\usepackage{epic}
\numberwithin{equation}{section}
\usepackage[margin=2.9cm]{geometry}
\usepackage{epstopdf}
\usepackage{xfrac}
\usepackage{float}
\usepackage[shortlabels]{enumitem}
\usepackage{mathtools}
\usepackage{amsmath,gauss,tikz}
\usepackage{lscape}
\usepackage{array}
\usepackage{setspace}
\usepackage{enumitem}

\theoremstyle{plain}
\newtheorem{theo}{Theorem}[section]
\newtheorem{cor}[theo]{Corollary}
\newtheorem{lem}[theo]{Lemma}

\newtheorem{prop}[theo]{Proposition}
\newtheorem{conj}[theo]{Conjecture}

\theoremstyle{definition}
\newtheorem{defn}[theo]{Definition}
\newtheorem{ex}[theo]{Example}
\newtheorem{rem}[theo]{Remark}

\title{Basis Criteria for Generalized Spline Modules via Determinant}

\author{Selma Altinok \and Samet Sarioglan}

\address{Selma Altinok, Hacettepe University Department of Mathematics, 06800 Beytepe Ankara Turkey.}
\email{sbhupal@hacettepe.edu.tr}

\address{Samet Sarioglan (Corresponding author), Hacettepe University Department of Mathematics, 06800 Beytepe Ankara Turkey.}
\email{ssarioglan@hacettepe.edu.tr}

\begin{document}
\begin{abstract} Given a graph whose edges are labeled by ideals of a commutative ring $R$ with identity, a generalized spline is a vertex labeling by the elements of $R$ such that the difference of the labels on adjacent vertices lies in the ideal associated to the edge. The set of generalized splines has a ring and an $R$-module structure. We study the module structure of generalized splines where the base ring is a greatest common divisor domain. We give basis criteria for generalized splines on cycles, diamond graphs and trees by using determinantal techniques. In the last section of the paper, we define a graded module structure for generalized splines and give some applications of the basis criteria for cycles, diamond graphs and trees.
\end{abstract}

\maketitle

\section{Introduction}
\label{intro}

A classical spline is a collection of polynomials defined on the faces of a polyhedral complex that agree on the intersection of adjacent faces. Classical splines are important tools in approximation theory, numerical analysis, computer graphics and numerical solutions of partial differential equations. Two main problems of the classical spline theory are computing the dimension and finding explicit bases of the vector space of splines up to some degree. Algebraic structure of classical splines is studied by many mathematicians as Billera~\cite{Bil1,Bil2,Bil}, Rose~\cite{Ros1,Ros2} and Schenck~\cite{Sch}. Billera~\cite{Bil1} introduced homological algebraic methods to solve the dimension problem. In~\cite{Sch}, Schenck also used homological algebra to give freeness criteria for the module structure of classical splines. In~\cite{Bil2}, Billera and Rose presented a description of classical splines in terms of dual graph of a polyhedral complex, which leads to generalized spline theory.

Let $R$ be a commutative ring with identity, $G = (V,E)$ be a graph and $\alpha : E \to \{ \text{ideals in }R \}$ be a function that labels edges of $G$ by ideals of $R$. A generalized spline on an edge labeled graph $(G, \alpha)$ is a vertex labeling $F \in R ^{|V|}$ such that for each edge $uv$, the difference $f_u - f_v \in \alpha(uv)$ where $f_u$ and $f_v$ denotes the vertex labels on $u$ and $v$ respectively. The set of all generalized splines on $(G,\alpha)$ over $R$ is denoted by $R_{(G,\alpha)}$.  The set $R_{(G,\alpha)}$ has a ring and $R$-module structure.

Gilbert, Polster and Tymoczko~\cite{Gil} introduced generalized spline theory and showed that if $R$ is a domain then the rank of $R_{(G,\alpha)}$ is equal to $\vert V \vert$. When $R$ is not a domain, Bowden and Tymoczko~\cite{Tym} proved that for a fixed number of vertices $\vert V \vert$, one can find examples of edge labeled graphs $(G, \alpha)$ where $R_{(G, \alpha)}$ has rank $n$ for all $2 \leq n \leq \vert V \vert$. In~\cite{Hand}, Handschy and the others focused on integer generalized splines on cycles. They presented a special type of generalized splines called flow-up classes and showed the existence of smallest flow-up classes on cycles. They also proved that flow-up classes with smallest leading entries form a basis for integer generalized spline modules on cycles. Same argument is proved for arbitrary graphs by Bowden and the others~\cite{Bow}. They also defined two new bases for integer generalized splines on cycles and studied the ring structure of $R_{(G, \alpha)}$. In~\cite{Alt}, we proved the existence of flow-up bases for generalized spline modules on arbitrary graphs over principal ideal domains. If $R$ is not a PID, there may not be a flow-up basis for $R_{(G, \alpha)}$ even it is free; see~\cite{Alt} for an example. Philbin and the others~\cite{Phi} gave an algorithm to produce a minimum generating set for $(\mathbb{Z} / m \mathbb{Z})_{(G, \alpha)}$ as a $\mathbb{Z}$-module. They also extended their algorithm to generalized splines over $\mathbb{Z}$ and gave a method to construct a $\mathbb{Z}$-module basis for $\mathbb{Z}_{(G,\alpha)}$. In~\cite{Dip}, DiPasquale introduced homological algebraic methods in generalized spline theory to give a freeness criteria for $R_{(G,\alpha)}$ under some conditions by using some results of Schenck~\cite{Sch}. DiPasquale also used generalized splines to get some results for the module of derivations of a graphic multi-arrangement.

In this paper we focus on the problem: When does a given set of generalized splines form a basis for $R_{(G, \alpha)}$? We study generalized splines over greatest common divisor domains. A GCD domain $R$ is an integral domain such that any two elements of $R$ have a greatest common divisor. In~\cite{Gjo}, Gjoni studied integer generalized splines on cycles in a senior project supervised by Rose and gave basis criteria for $\mathbb{Z}_{(C_n , \alpha)}$ via determinant of flow-up classes. Their method does not work in general since the existence of flow-up bases is not guaranteed when $R$ is not a PID. We generalize their work to the case where $R$ is a GCD domain. Mahdavi~\cite{Mah} studied integer generalized splines on diamond graph $D_{3,3}$ with Rose and they obtained a partial result under some conditions for basis criteria for $\mathbb{Z}_{(D_{3,3} , \alpha)}$. In this paper we give a complete proof of basis criteria for $R_{(D_{3,3} , \alpha)}$ over any GCD domain. They also conjectured that their result can be generalized to any arbitrary diamond graph $D_{m,n}$. We give a proof of their conjecture and mention some other generalizations of their statement. We also give basis criteria for $R_{(G, \alpha)}$ on any tree over any GCD domain by using determinantal techniques and flow-up bases.

Finally, we define the homogenization  ${\hat R}_{(G, \hat{\alpha})}$ of $R_{(G,\alpha)}$ which is a graded $\hat R$-module, and investigate a freeness relation between these two modules. We also introduce some applications of the basis criteria for cycles, diamond graphs and trees on the graded module structure of generalized splines.

\section{Generalized Splines}
\label{Gsplines}

In this section, we introduce some basic definitions and properties of generalized splines.

\begin{defn} Given a graph $G$ and a commutative ring $R$ with identity, an edge labeling of $G$ is a function $\alpha : E \to \{ \text{ideals in }R \}$ that labels each edge of $G$ by an ideal of $R$. A generalized spline on an edge labeled graph $(G, \alpha)$ is a vertex labeling $F \in R ^{|V|}$ such that for each edge $uv$, the difference $f_u - f_v \in \alpha(uv)$ where $f_u$ denotes the label on vertex $u$. The collection of all generalized splines on a base ring $R$ over the edge labeled graph $(G, \alpha)$ is denoted by $R_{(G,\alpha)}$.
\end{defn}

Throughout the paper we assume that the base ring $R$ is a GCD domain. Each edge of $(G,\alpha)$ is labeled with a generator of the ideal $I$ if the corresponding ideal $I$ is principal. From now on we refer to generalized splines as splines. Let $(G, \alpha)$ be an edge labeled graph with $n$ vertices. We denote the elements of $R_{(G,\alpha)}$ by column matrix notation with entries in order from bottom to top as follows:
\begin{displaymath}
F = \begin{bmatrix} f_n \\ \vdots \\ f_1 \end{bmatrix} \in R_{(G,\alpha)}.
\end{displaymath}
We also use vector notation as $F = (f_1 , \ldots , f_n)$. 

\begin{ex} Let $(G, \alpha)$ be as the figure below.
\begin{figure}[H]
\begin{center}
\scalebox{0.16}{\includegraphics{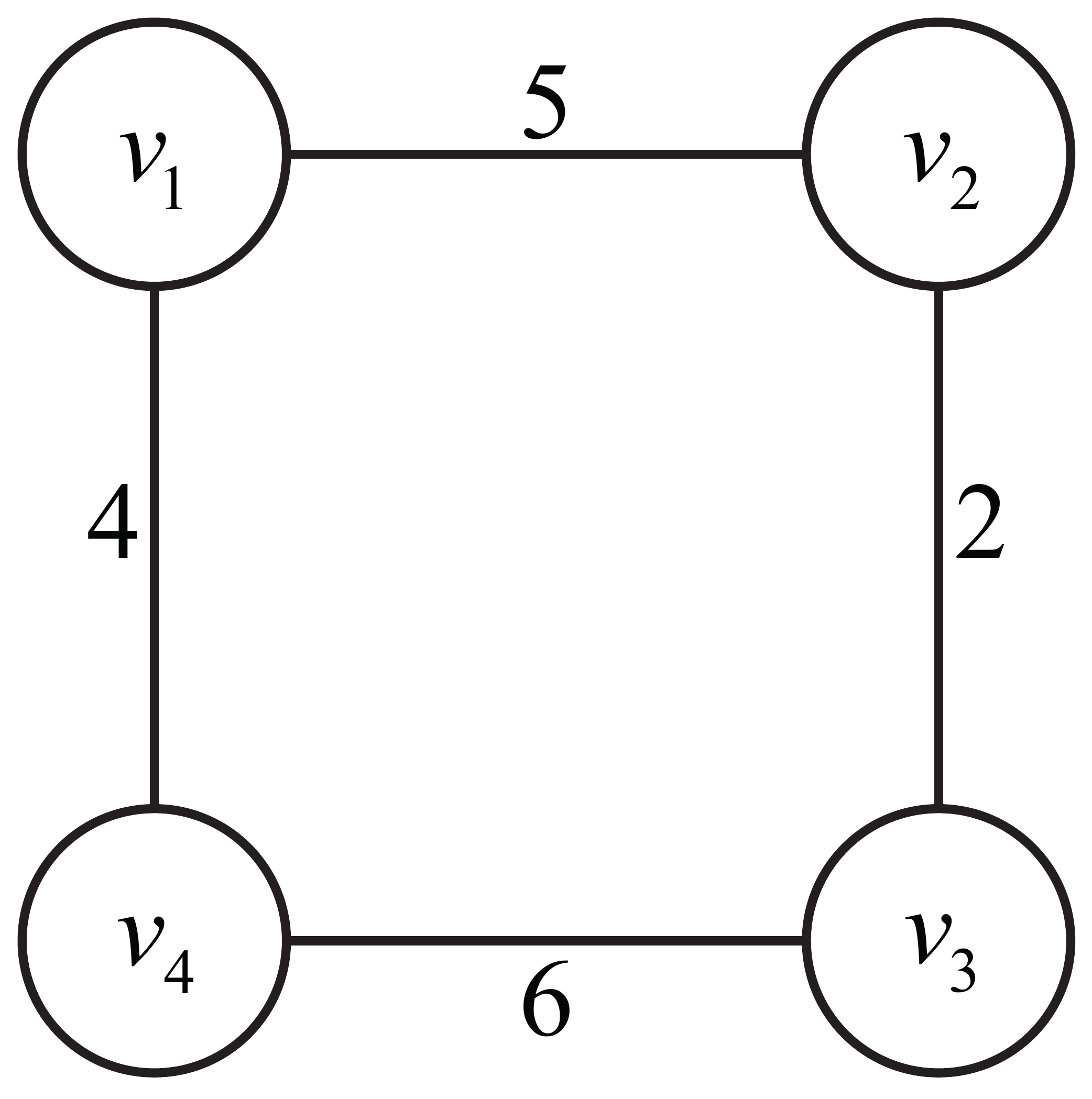}}
\caption{Example of spline}
\label{det1}
\end{center}
\end{figure}
\noindent
A spline over $(G, \alpha)$ can be given by $F = (2,12,14,26)$. 
\end{ex}

The following lemma states that $R_{(G,\alpha)}$ does not depend on the ordering of vertices of $(G, \alpha)$:

\begin{lem} Let $(G,\alpha)$ be an edge labeled graph with $n$ vertices and let $(G',\alpha)$ be the edge labeled graph obtained by reordering the vertices of $(G,\alpha)$ by a permutation $\sigma \in \mathcal{S}_n$. Then $R_{(G,\alpha)} \cong R_{(G',\alpha)}$.
\label{reorder}
\end{lem}

\begin{proof} We show that $\sigma$ induces an $R$-module isomorphism  $\sigma ^{\star}: R_{(G,\alpha)} \to R_{(G',\alpha)}$ by reordering the components of a spline $F \in R_{(G,\alpha)}$ as
\begin{displaymath}
\begin{array}{ccl}
\sigma ^{\star} : R_{(G,\alpha)} &\to& R_{(G',\alpha)} \\
F = (f_1 , \ldots , f_n ) &\to& (f_{\sigma(1)} , \ldots , f_{\sigma(n)}).
\end{array}
\end{displaymath}
In order to see that $\sigma ^{\star}(F) \in R_{(G',\alpha)}$, let $v_i , v_j$ be two adjacent vertices of $R_{(G',\alpha)}$. By the definition of $\sigma ^{\star}$, $\sigma ^{\star}(F)_i = f_{\sigma(i)}$ and $\sigma ^{\star}(F)_j = f_{\sigma(j)}$ are connected by the same edge $e_{ij}$ on both $R_{(G,\alpha)}$ and $R_{(G',\alpha)}$. Since $F \in R_{(G,\alpha)}$, we have $f_{\sigma(i)} - f_{\sigma(j)} \in \alpha(e_{ij})$. Thus $\sigma ^{\star}(F)_i - \sigma ^{\star}(F)_j \in \alpha(e_{ij})$ and so $\sigma ^{\star}(F) \in R_{(G',\alpha)}$. 
\end{proof}

The following example illustrates the reordering operation:

\begin{ex} Let $(G,\alpha)$ be as in the figure below and $\sigma = (13524) \in \mathcal{S}_5$.

\begin{figure}[H]
\begin{center}
\scalebox{0.15}{\includegraphics{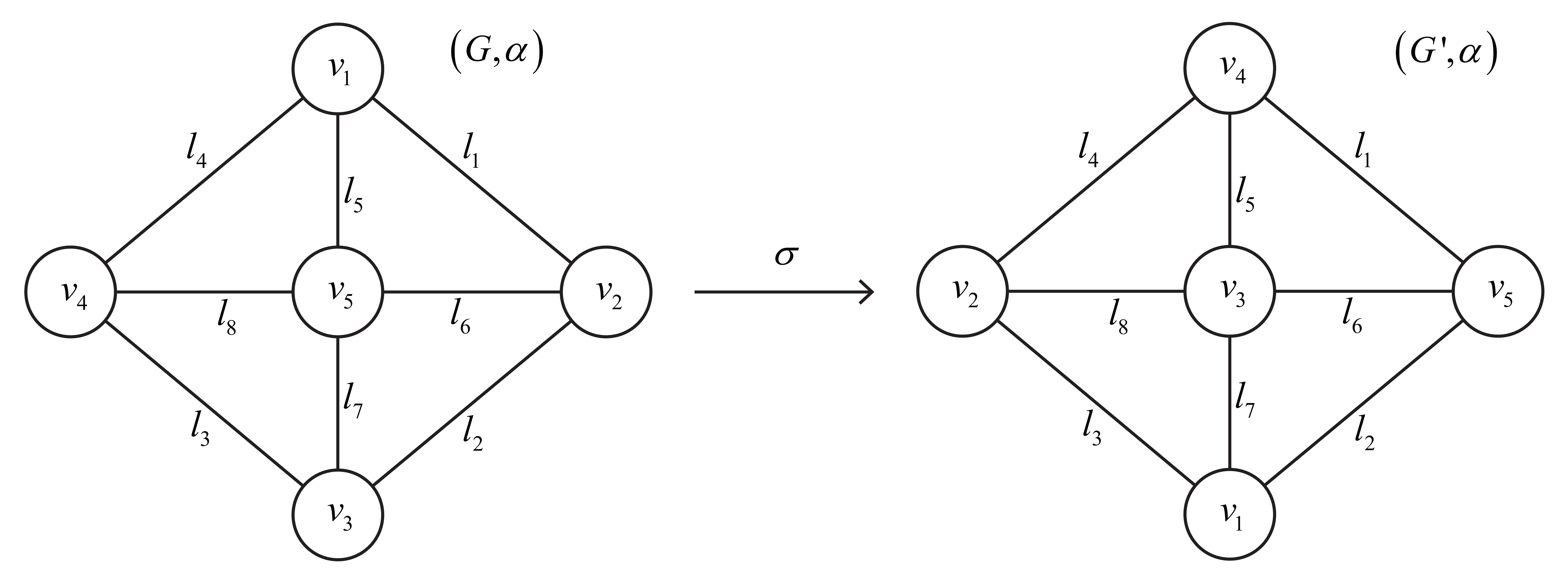}}
\caption{Edge labeled graphs $(G,\alpha)$ and $(G',\alpha)$}
\label{sp73}
\end{center}
\end{figure}
\noindent
Here we have
\begin{displaymath}
\begin{gathered}
\sigma ^{\star} : R_{(G,\alpha)} \to R_{(G',\alpha)} \\
\sigma ^{\star}  (f_1 , f_2 , f_3 , f_4 , f_5) = (f_3 , f_4 , f_5 , f_1 , f_2).
\end{gathered}
\end{displaymath}
\end{ex}

A special type of splines, which is called flow-up classes, is a useful tool to find $R$-module bases for $R_{(G,\alpha)}$.

\begin{defn} Let $(G, \alpha)$ be an edge labeled graph with $n$ vertices. Fix $i$ with $1 \leq i \leq n$. A flow-up class $F^{(i)}$ is a spline in $R_{(G,\alpha)}$ with the components $F^{(i)}_i \neq 0$ and $F^{(i)}_j = 0$ for all $j < i$.
\end{defn}

\begin{ex} Consider the edge labeled graph $(G, \alpha)$ in Figure~\ref{det1} again. Flow-up classes on $(G, \alpha)$ can be given by $F^{(1)} = (1,1,1,1)$, $F^{(2)} = (0,10,0,0)$, $F^{(3)} = (0,0,2,0)$ and $F^{(4)} = (0,0,0,12)$.
\end{ex}

We can set $F^{(1)} = (1, \ldots ,1)$ for any graph. In order to see the existence of $F^{(i)}$ with $i > 1$ on any graph, label $F^{(i)}_i$ by the product of all edge labels on $(G, \alpha)$ and label other vertices by zero. Bowden and the others~\cite{Bow} proved that flow-up classes with smallest leading entries form a module basis for $R_{(G,\alpha)}$ where $R$ is an integral domain. In~\cite{Alt}, we proved the existence of flow-up bases on any graphs over principal ideal domains. If $R$ is not a domain, then $R_{(G,\alpha)}$ may not have a flow-up basis even it is free.

In the next section we begin to discuss determinantal techniques for splines.

\section{Determinant}
\label{determinant}

Let $(G, \alpha)$ be an edge labeled graph with $n$-vertices. Let $A = \{ F_1 , \ldots , F_n \} \subset R_{(G,\alpha)}$ with $F_i = (f_{i1} , \ldots , f_{in})$. We can rewrite $A$ in a matrix form, whose columns are the elements of $A$ such as
\begin{displaymath}
A = \begin{bmatrix} f_{1n} & f_{2n} & \ldots & f_{nn} \\ \vdots \\ f_{12} & f_{22} & \ldots & f_{n2} \\ f_{11} & f_{21} & \ldots & f_{n1} \end{bmatrix}.
\end{displaymath}
The determinant $\big\vert A \big\vert$ is denoted by $\big\vert F_1 \text{ } F_2 \text{ } \ldots \text{ } F_n \big\vert$. We will give basis criteria for $R_{(G,\alpha)}$ by using this determinant.

\begin{prop} Let $(G, \alpha)$ be an edge labeled graph with $n$-vertices. Let $\{ F_1 , \ldots , F_n \}$ forms a basis for $R_{(G,\alpha)}$ and let $\{ G_1 , \ldots , G_n \} \subset R_{(G,\alpha)}$. Then $\big\vert F_1 \text{ } F_2 \text{ } \ldots \text{ } F_n \big\vert$ divides $\big\vert G_1 \text{ } G_2 \text{ } \ldots \text{ } G_n \big\vert$.
\label{detprop2}
\end{prop}

\begin{proof} See Lemma 5.1.4. in~\cite{Gjo}.
\end{proof}

\begin{cor} Let $(G, \alpha)$ be an edge labeled graph with $n$-vertices. Let $\{ F_1 , \ldots , F_n \}$ forms a basis for $R_{(G,\alpha)}$. If $\{ G_1 , \ldots , G_n \} \subset R_{(G,\alpha)}$ is another basis for $R_{(G,\alpha)}$, then $\big\vert F_1 \text{ } F_2 \text{ } \ldots \text{ } F_n \big\vert = r \big\vert G_1 \text{ } G_2 \text{ } \ldots \text{ } G_n \big\vert$ where $r \in R$ is a unit.
\label{detcor3}
\end{cor}

The following lemma shows the relation between the determinant of a basis of $R_{(G,\alpha)}$ and $R_{(G',\alpha)}$ where $G'$ is obtained by reordering the vertices of $G$:

\begin{lem} Let $(G, \alpha)$ be an edge labeled graph with $n$-vertices and let $\{ F_1 , \ldots , F_n \}$ forms a basis for $R_{(G,\alpha)}$. Let $\sigma \in \mathscr{S}_n$ be a permutation and let $\sigma \big( (G, \alpha) \big) = (G' , \alpha)$ be a vertex reordering of $(G, \alpha)$ as defined in Lemma~\ref{reorder}. If $\{ G_1 , \ldots , G_n \}$ is a basis for $R_{(G',\alpha)}$, then $\big\vert F_1 \text{ } F_2 \text{ } \ldots \text{ } F_n \big\vert = r \big\vert G_1 \text{ } G_2 \text{ } \ldots \text{ } G_n \big\vert$ where $r \in R$ is a unit.
\label{reorderlem}
\end{lem}

\begin{proof} Reordering the vertices corresponds to replacing the rows of the spline matrix, and does not change the determinant. Hence
\begin{displaymath}
\big\vert F_1 \text{ } F_2 \text{ } \ldots \text{ } F_n \big\vert = \pm \big\vert \sigma(F_1) \text{ } \sigma(F_2) \text{ } \ldots \text{ } \sigma(F_n) \big\vert.
\end{displaymath}
Here $\{ \sigma(F_1) , \ldots , \sigma(F_n) \} \subset R_{(G',\alpha)}$ and since $\{ G_1 , \ldots , G_n \}$ is a basis for $R_{(G',\alpha)}$, $\big\vert G_1 \text{ } G_2 \text{ } \ldots \text{ } G_n \big\vert$ divides $\big\vert \sigma(F_1) \text{ } \sigma(F_2) \text{ } \ldots \text{ } \sigma(F_n) \big\vert = \big\vert F_1 \text{ } F_2 \text{ } \ldots \text{ } F_n \big\vert$ by Proposition~\ref{detprop2}.

Now consider $\sigma^{-1} \in \mathscr{S}_n$. Then
\begin{displaymath}
\big\vert G_1 \text{ } G_2 \text{ } \ldots \text{ } G_n \big\vert = \pm \big\vert \sigma^{-1} (G_1) \text{ } \sigma^{-1} (G_2) \text{ } \ldots \text{ } \sigma^{-1} (G_n) \big\vert
\end{displaymath}
as explained above. Also we have $\{ \sigma^{-1} (G_1) , \ldots , \sigma^{-1} (G_n) \} \subset R_{(G,\alpha)}$ and $\big\vert \sigma(F_1) \text{ } \sigma(F_2) \text{ } \ldots \text{ } \sigma(F_n) \big\vert$ divides $\big\vert \sigma^{-1} (G_1) \text{ } \sigma^{-1} (G_2) \text{ } \ldots \text{ } \sigma^{-1} (G_n) \big\vert = \big\vert G_1 \text{ } G_2 \text{ } \ldots \text{ } G_n \big\vert$ by Proposition~\ref{detprop2}. Hence we conclude that $\big\vert F_1 \text{ } F_2 \text{ } \ldots \text{ } F_n \big\vert = r \big\vert G_1 \text{ } G_2 \text{ } \ldots \text{ } G_n \big\vert$ where $r \in R$ is a unit.
\end{proof}

Throughout the rest of the paper we focus on to give basis criteria via determinant for spline modules  $R_{(G,\alpha)}$ on cycles, diamond graphs and trees. In order to do this, we define a crucial element $Q_G \in R$ by using zero trials, which are discussed in~\cite{Alt}. Let $(G,\alpha)$ be an edge labeled graph with $k$ vertices. Fix a vertex $v_i$ on $(G,\alpha)$ with $i \geq 2$. Label all vertices $v_j$ with $j < i$ by zero. By using the notations in~\cite{Alt}, we define $Q_G$ as
\begin{displaymath}
Q_G = \prod\limits_{i=2} ^k \Big[ \Big\{ \big( \textbf{p}_t ^{(i,0)} \big) \text{ $\big\vert$ } t = 1,\ldots, m_i \Big\} \Big]
\end{displaymath}
where $m_i$ is the number of the zero trials of $v_i$. The element $Q_G$ can be formularized in terms of edge labels on cycles, diamond graphs and trees. In general, this is not easy. Gjoni~\cite{Gjo} and Mahdavi~\cite{Mah} studied integer splines on cycles and diamond graphs respectively and they stated that a given set of splines forms a basis for $\mathbb{Z}_{(G,\alpha)}$ if and only if the determinant of the matrix whose columns are the elements of the given set is equal to a formula $Q$  given by edge labels. We will show that the formula $Q$ corresponds to $Q_G$ and generalize their statement.

\subsection{Determinant of Splines on Cycles}

In~\cite{Gjo}, Gjoni gave basis criteria for integer splines on cycles by using determinantal techniques. Gjoni used flow-up bases to prove Theorem~\ref{gjotheo}. In general his approach does not work since the existence of flow-up bases is not guaranteed when $R$ is not a PID. Such an example can be found in~\cite{Alt}. In this section we generalize Theorem~\ref{gjotheo} to any GCD domain. We give the statement of Gjoni below.

\begin{theo} \emph{\cite{Gjo}}
Fix the edge labels on $(C_n , \alpha)$. Let 
\begin{displaymath}
Q = \dfrac{l_1 l_2 \cdots l_n}{\big( l_1 , l_2 , \ldots , l_n \big)}.
\end{displaymath}
 and let $ F_1 , \ldots , F_n \in \mathbb{Z}_{(C_n , \alpha)}$ . Then $\{ F_1 , \ldots , F_n \}$ forms a module basis for $\mathbb{Z}_{(C_n , \alpha)}$ if and only if $\big\vert F_1 \text{ } F_2 \text{ } \ldots \text{ } F_n \big\vert = \pm Q$.
\label{gjotheo}
\end{theo}

The following lemma shows that the formula $Q$ above is equal to $Q_{C_n}$:

\begin{lem} Let $(C_n , \alpha)$ be an edge labeled $n$-cycle. Then
\begin{displaymath}
Q_{C_n} = \dfrac{l_1 l_2 \cdots l_n}{\big( l_1 , l_2 , \ldots , l_n \big)}.
\end{displaymath}
\end{lem}

\begin{proof}
\begin{displaymath}
\begin{array}{ccl}
Q_{C_n} &=& \big[ l_1 , (l_2 , \ldots , l_n) \big] \cdot \big[ l_2 , (l_3 , \ldots , l_n) \big] \cdots \big[ l_{n-2} , (l_{n-1} , l_n) \big] \cdot \big[ l_{n-1} , l_n \big] \vspace{.3cm} \\
&=& \dfrac{l_1 (l_2 , l_3 \ldots , l_n)}{(l_1 , l_2 , \ldots , l_n)} \cdot \dfrac{l_2 (l_3 , \ldots , l_n)}{(l_2 , l_3 , \ldots , l_n)} \cdots \dfrac{l_{n-2} (l_{n-1} , l_n)}{(l_{n-2} , l_{n-1} , l_n)} \cdot \big[ l_{n-1} , l_n \big] = \dfrac{l_1 l_2 \cdots l_n}{\big( l_1 , l_2 , \ldots , l_n \big)}.
\end{array}
\end{displaymath}
\end{proof}

In order to show that Theorem~\ref{gjotheo} holds also when the base ring $R$ is a GCD domain, we first need some lemmas.

\begin{lem} Let $(C_n , \alpha)$ be an edge labeled $n$-cycle. If we set $\hat{l_i} = l_1 \cdots l_{i-1} l_{i+1} \cdots l_n$, then $\hat{l_i}$ divides $\big\vert F_1 \text{ } F_2 \text{ } \ldots \text{ } F_n \big\vert$ for all $i = 1, \ldots , n$.
\label{cyclem}
\end{lem}

\begin{proof} See Lemma 5.1.1. in~\cite{Gjo}.
\end{proof}

\begin{lem} Let $(C_n , \alpha)$ be an edge labeled $n$-cycle. Let $\{ F_1 , \ldots , F_n \} \subset R_{(C_n , \alpha)}$. Then $Q_{C_n}$ divides $\big\vert F_1 \text{ } F_2 \text{ } \ldots \text{ } F_n \big\vert$.
\label{cyclem1}
\end{lem}

\begin{proof} Since $\hat{l_i}$ divides $\big\vert F_1 \text{ } F_2 \text{ } \ldots \text{ } F_n \big\vert$ for all $i = 1, \ldots , n$ by Lemma~\ref{cyclem}, $\big[ \hat{l_1} , \hat{l_2} , \ldots , \hat{l_n} \big] = Q_{C_n}$ also divides $\big\vert F_1 \text{ } F_2 \text{ } \ldots \text{ } F_n \big\vert$.
\end{proof}

We give the main theorem of this section below which is a generalization of Theorem~\ref{gjotheo} to any GCD domain. One direction of Theorem~\ref{cycthm} can be proved exactly the same as Theorem 5.1.7. in~\cite{Gjo}. Other way around does not work as in Theorem 5.1.7. in~\cite{Gjo}. We use different techniques to prove it.

\begin{theo} Let $(C_n , \alpha)$ be an edge labeled $n$-cycle. Let $\{ F_1 , \ldots , F_n \} \subset R_{(C_n , \alpha)}$. Then $\{ F_1 , \ldots , F_n \}$ forms a basis for $R_{(C_n , \alpha)}$ if and only if $\big\vert F_1 \text{ } F_2 \text{ } \ldots \text{ } F_n \big\vert = r \cdot Q_{C_n}$ where $r \in R$ is a unit.
\label{cycthm}
\end{theo}

\begin{proof}  The proof of the second part of the theorem can be found in Theorem 5.1.7, ~\cite{Gjo}. For the proof of the first part, we assume that  $\{ F_1 , \ldots , F_n \}$ forms a basis for $R_{(C_n , \alpha)}$. Then the determinant $\big\vert F_1 \text{ } F_2 \text{ } \ldots \text{ } F_n \big\vert = r \cdot Q_{C_n}$ for some $r \in R$ by Lemma~\ref{cyclem1}. It suffices to prove that $r$ is a unit. Assume that $\big( l_1 , l_2 , \ldots , l_n \big) = a \neq 1$. Then $l_i = a \cdot l' _ i$ for all $i = 1, \ldots , n$ with $(l' _1 , \ldots , l' _n) = 1$. 

We construct matrices $A^{(i)} = \begin{bmatrix} A^{(i)}_0 & A^{(i)}_1 & \ldots & A^{(i)}_{n-1} \end{bmatrix}$ for all $i = 1, \ldots , n$ with columns $A^{(i)}_j$ where $j = 0, \ldots , n-1$. Let $\begin{bmatrix} A^{(i)}_j \end{bmatrix} _k$ denote the $k$-th entry of the column. Notice that the entries are ordered from bottom to top. Fix
\begin{displaymath}
A^{(i)}_0 = \begin{bmatrix} 1 \\ \vdots \\ 1 \end{bmatrix}
\end{displaymath}
for all $1 \leq i \leq n$. For a fixed $i$ with $1 \leq i \leq n$, define the entries of the columns $A^{(i)}_j$ for $j = 1, \ldots , n-1$ as follows:

\begin{itemize}
\item For $j < i$,
\begin{displaymath}
\begin{bmatrix} A^{(i)}_j \end{bmatrix} _k = 
\begin{cases} \big[ l_j , ( l_1 , \ldots , l_{j-1} , l_i , l_n ) \big] l' _i , & j < k \leq i \\
0, & \text{otherwise}.
\end{cases}
\end{displaymath}
\item For $j = i$,
\begin{displaymath}
\begin{bmatrix} A^{(i)}_j \end{bmatrix} _k = 
\begin{cases} \big[ l_i , l_n \big] l' _i, & 1 \leq k \leq i \\
0, & \text{otherwise}.
\end{cases}
\end{displaymath}
\item For $j > i$,
\begin{displaymath}
\begin{bmatrix} A^{(i)}_j \end{bmatrix} _k = 
\begin{cases} \big[ l_j , ( l_1 , \ldots , l_{j-1} , l_n ) \big] l' _i, & i < k \leq j \\
0, & \text{otherwise}.
\end{cases}
\end{displaymath}
\end{itemize}

\noindent
It is easy to see that each column of $A^{(i)}$ for $1 \leq i \leq n$ is a spline. In order to compute the determinant $\big\vert A^{(i)} \big\vert$ for $1 \leq i \leq  n$, we first obtain the following matrix by applying few column operations on $A^{(i)}$ if it is necessary:
\begin{displaymath}
A^{(i)} = 
\begin{bmatrix} A^{(i)}_0 & A^{(i)}_i & A^{(i)}_1 & A^{(i)}_2 & \ldots & A^{(i)}_{i-1} & A^{(i)}_{i+1} & \ldots & A^{(i)}_{n-1} \end{bmatrix}
\end{displaymath}
so that
\begin{displaymath}
\begin{array}{ccl}
\big\vert A^{(i)} \big\vert &=& \big[ l_i , l_n \big] \big[ l_1 , \big( l_i , l_n \big) \big] \big[ l_2 , \big( l_1 , l_i , l_n \big) \big] \cdots \big[ l_{n-1} , \big( l_1 , \ldots , l_{n-2} , l_n \big) \big] {l' _i}^{(n-1)} \\ &=& Q_{C_n} \cdot {l' _i}^{(n-1)}
\end{array}
\end{displaymath}
for all $i = 1, \ldots , n$. By Proposition~\ref{detprop2}, $ r \cdot Q_{C_n}=\big\vert F_1 \text{ } F_2 \text{ } \ldots \text{ } F_n \big\vert$ divides $\big\vert A^{(i)} \big\vert = {l' _i}^{(n-1)} \cdot Q_{C_n}$ and so $r$ divides ${l' _i}^{(n-1)}$ for all $i = 1, \ldots , n$. Then $r$ divides $\big( {l' _1}^{(n-1)} , \ldots , {l' _n}^{(n-1)} \big) = 1$. Hence $r$ is a unit.

If $\big( l_1 , l_2 , \ldots , l_n \big) = 1$, then there exists a coprime pair $l_i , l_j$. In this case we can construct matrices $A^{(i)}$ and $A^{(j)}$ such that $\big\vert A^{(i)} \big\vert = Q_{C_n} \cdot {l}^{(n-1)} _i$ and $\big\vert A^{(j)} \big\vert = Q_{C_n} \cdot {l}^{(n-1)} _j$. Hence by the same observation above we conclude that $r$ divides $({l}^{(n-1)} _i , {l}^{(n-1)} _j) = 1$ . Thus $r$ is a unit.
\end{proof}

The following example is an application of Theorem~\ref{cycthm}:

\begin{ex} Consider the edge labeled graph $(C_5 , \alpha)$. 

\begin{figure}[H]
\begin{center}
\scalebox{0.16}{\includegraphics{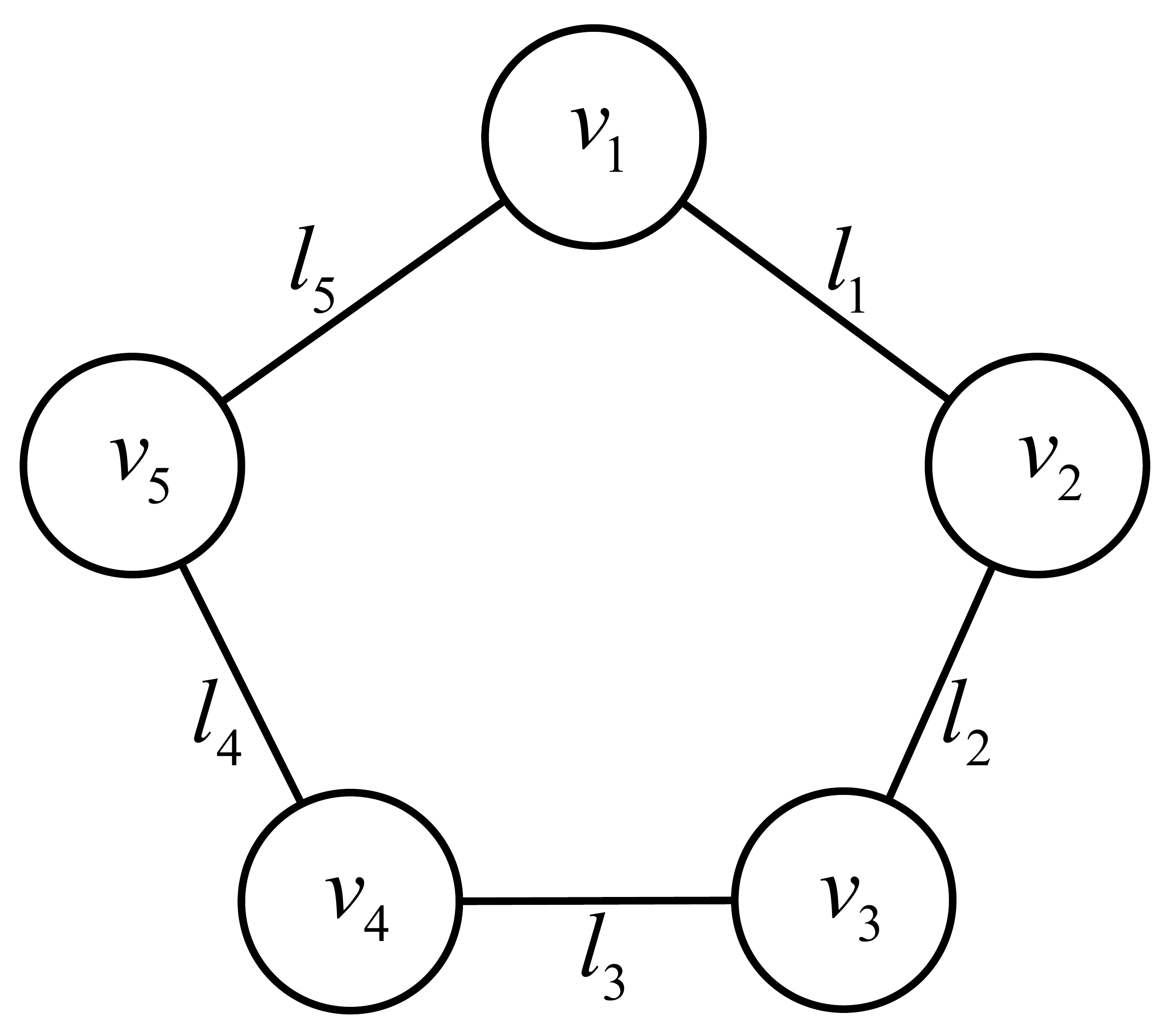}}
\caption{Edge labeled graph $(C_5 , \alpha)$}
\label{c5}
\end{center}
\end{figure}
\noindent
Let $\{ F_1 , \ldots , F_5 \}$ be a basis for $R_{(C_5 , \alpha)}$. By Lemma~\ref{cyclem1}, $\big\vert F_1 \text{ } F_2 \text{ } \ldots \text{ } F_5 \big\vert = r \cdot Q_{C_5}$ for some $r \in R$ where $Q_{C_5} = \dfrac{l_1 \cdots l_5}{(l_1 , \ldots l_5)}$. Assume that $(l_1 , \ldots , l_5) = a$. Then $l_i = a \cdot l' _i$ for all $i = 1, \ldots , 5$ and $(l' _1 , \ldots , l' _5) = 1$. Consider the following matrices:
\begin{displaymath}
A^{(1)} = \begin{bmatrix} 
1 & 0 & 0 & 0 & 0 \\
1 & 0 & 0 & 0 & \big[ l_4 , \big( l_1 , l_2 , l_3 , l_5 \big) \big] l' _1 \\
1 & 0 & 0 & \big[ l_3 , \big( l_1 , l_2 , l_5 \big) \big] l' _1 & \big[ l_4 , \big( l_1 , l_2 , l_3 , l_5 \big) \big] l' _1 \\
1 & 0 & \big[ l_2 , \big( l_1 , l_5 \big) \big] l' _1 & \big[ l_3 , \big( l_1 , l_2 , l_5 \big) \big] l' _1 & \big[ l_4 , \big( l_1 , l_2 , l_3 , l_5 \big) \big] l' _1 \\
1 & \big[ l_1 , l_5 \big] l' _1 & 0 & 0 & 0
\end{bmatrix},
\end{displaymath}
\begin{displaymath}
A^{(2)} = \begin{bmatrix} 
1 & 0 & 0 & 0 & 0 \\
1 & 0 & 0 & 0 & \big[ l_4 , \big( l_1 , l_2 , l_3 , l_5 \big) \big] l' _2 \\
1 & 0 & 0 & \big[ l_3 , \big( l_1 , l_2 , l_5 \big) \big] l' _2 & \big[ l_4 , \big( l_1 , l_2 , l_3 , l_5 \big) \big] l' _2 \\
1 & \big[ l_1 , \big( l_2 , l_5 \big) \big] l' _2 & \big[ l_2 , l_5 \big] l' _2 & 0 & 0 \\
1 & 0 & \big[ l_2 , l_5 \big] l' _2 & 0 & 0 \\
\end{bmatrix},
\end{displaymath}
\begin{displaymath}
A^{(3)} = \begin{bmatrix} 
1 & 0 & 0 & 0 & 0 \\
1 & 0 & 0 & 0 & \big[ l_4 , \big( l_1 , l_2 , l_3 , l_5 \big) \big] l' _3 \\
1 & \big[ l_1 , \big( l_3 , l_5 \big) \big] l' _3 & \big[ l_2 , \big( l_1 , l_3 , l_5 \big) \big] l' _3 & \big[ l_3 , l_5 \big] l' _3 & 0 \\
1 & \big[ l_1 , \big( l_3 , l_5 \big) \big] l' _3 & 0 & \big[ l_3 , l_5 \big] l' _3 & 0 \\
1 & 0 & 0 & \big[ l_3 , l_5 \big] l' _3 & 0 \\
\end{bmatrix},
\end{displaymath}
\begin{displaymath}
A^{(4)} = \begin{bmatrix} 
1 & 0 & 0 & 0 & 0 \\
1 & \big[ l_1 , \big( l_4 , l_5 \big) \big] l' _4 & \big[ l_2 , \big( l_1 , l_4 , l_5 \big) \big] l' _4 & \big[ l_3 , \big( l_1 , l_2 , l_4 , l_5 \big) \big] l' _4 & \big[ l_4 , l_5 \big] l' _4 \\
1 & \big[ l_1 , \big( l_4 , l_5 \big) \big] l' _4 & \big[ l_2 , \big( l_1 , l_4 , l_5 \big) \big] l' _4 & 0 & \big[ l_4 , l_5 \big] l' _4 \\
1 & \big[ l_1 , \big( l_4 , l_5 \big) \big] l' _4 & 0 & 0 & \big[ l_4 , l_5 \big] l' _4 \\
1 & 0 & 0 & 0 & \big[ l_4 , l_5 \big] l' _4 \\
\end{bmatrix},
\end{displaymath}
\begin{displaymath}
A^{(5)} = \begin{bmatrix} 
1 & \big[ l_1 , l_5 \big] l' _5 & \big[ l_2 , \big( l_1 , l_5 \big) \big] l' _5 & \big[ l_3 , \big( l_1 , l_2 ,  l_5 \big) \big] l' _5 & \big[ l_4 , \big( l_1 , l_2 , l_3 , l_5 \big) \big] l' _5 \\
1 & \big[ l_1 , l_5 \big] l' _5 & \big[ l_2 , \big( l_1 , l_5 \big) \big] l' _5 & \big[ l_3 , \big( l_1 , l_2 , l_5 \big) \big] l' _5 & 0 \\
1 & \big[ l_1 , l_5 \big] l' _5 & \big[ l_2 , \big( l_1 , l_5 \big) \big] l' _5 & 0 & 0 \\
1 & \big[ l_1 , l_5 \big] l' _5 & 0 & 0 & 0 \\
1 & 0 & 0 & 0 & 0 \\
\end{bmatrix}.
\end{displaymath}
Each column of $A^{(i)} $ is an element of $R_{(C_5 , \alpha)}$. By Proposition~\ref{detprop2}, $\big\vert F_1 \text{ } F_2 \text{ } \ldots \text{ } F_5 \big\vert = r \cdot Q_{C_5}$ divides $\big\vert A^{(i)} \big\vert = {l'_i}^4 \cdot Q_{C_5}$. Hence $r$ divides ${l'_i}^4$ and so $r$ divides $\big( {l'_1}^4  , \ldots , {l'_5}^4 \big) = 1$. Thus $r$ is a unit.
\end{ex}

\subsection{\normalsize{Determinant of Splines on Diamond Graph $D_{3,3}$}} \label{secd33}

Diamond graph $D_{m,n}$ is obtained by gluing two cycles $C_m$ and $C_n$ along a common edge. The following figure illustrates the diamond graph $D_{3,3}$:

\begin{figure}[H]
\begin{center}
\scalebox{0.16}{\includegraphics{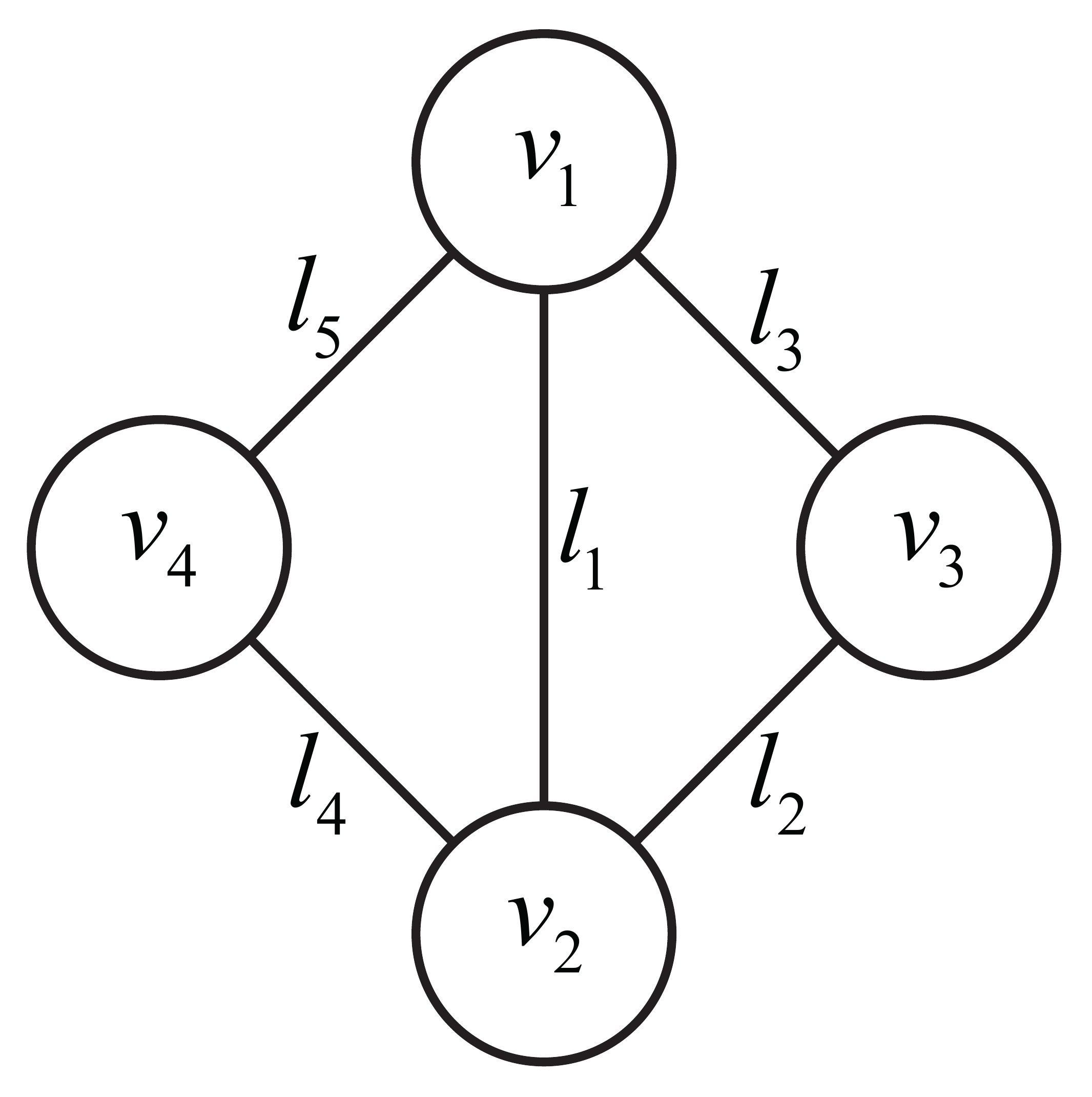}}
\caption{Edge labeled diamond graph $(D_{3,3} , \alpha)$}
\label{d33}
\end{center}
\end{figure}

In~\cite{Mah}, Mahdavi tried to give basis criteria for $\mathbb{Z}_{(D_{3,3}, \alpha)}$ with Rose. They proved a similar result as Lemma~\ref{cyclem1} for diamond graphs under some conditions. The statement is given below.

\begin{lem} \emph{\cite{Mah}} Fix the edges on $(D_{3,3} , \alpha)$. Let $(l_2 , l_3 , l_4 , l_5) = (l_1 , l_2) = (l_1 , l_3) = (l_1 , l_4) = (l_1 , l_5) = 1$, and $Q = \dfrac{l_1 l_2 l_3 l_4 l_5}{\big( (l_2 , l_3) (l_4 , l_5) , l_1 (l_2 , l_3 , l_4 , l_5) \big)}$. If $W,X,Y,Z \in \mathbb{Z}_{(D_{3,3}, \alpha)}$, then $Q$ divides $\big\vert W \text{ } X \text{ } Y \text{ } Z \big\vert$.
\label{mahlem}
\end{lem}

They also used the existence of flow-up bases for diamond graphs to prove Lemma~\ref{mahlem}.  Their proof does not work in general, especially if $R$ is not PID. They gave the following conjecture for $\mathbb{Z}_{(D_{3,3}, \alpha)}$:

\begin{conj} \emph{\cite{Mah}}
Fix the edge labels on $(D_{3,3} , \alpha)$. Let 
\begin{displaymath}
Q = \dfrac{l_1 l_2 l_3 l_4 l_5}{\big( (l_2 , l_3) (l_4 , l_5) , l_1 (l_2 , l_3 , l_4 , l_5) \big)}
\end{displaymath}
and let $W,X,Y,Z \in \mathbb{Z}_{(D_{3,3}, \alpha)}$. If $\big\vert W \text{ } X \text{ } Y \text{ } Z \big\vert = \pm Q$, then $\{ W,X,Y,Z \}$ forms a basis for $\mathbb{Z}_{(D_{3,3}, \alpha)}$.
\label{conjec}
\end{conj}

The following lemma states that the formula $Q$ above is equal to $Q_{D_{3,3}}$:

\begin{lem} Let $D_{3,3}$ be as in Figure~\ref{d33}. Then
\begin{displaymath}
Q_{D_{3,3}} = \dfrac{l_1 l_2 l_3 l_4 l_5}{\big( (l_2 , l_3) (l_4 , l_5) , l_1 (l_2 , l_3 , l_4 , l_5) \big)}.
\end{displaymath}
\label{33lemm}
\end{lem}

\begin{proof}
\begin{displaymath}
\begin{array}{ccl}
Q_{D_{3,3}} &=& \big[ l_1 , (l_2 , l_3) , (l_4 , l_5) \big] \cdot [l_2 , l_3] \cdot [l_4 , l_5] \vspace{.3cm} \\ 
&=& \dfrac{l_1 (l_2 , l_3) (l_4 , l_5)}{(l_1 (l_2,l_3) , l_1 (l_4 , l_5) , (l_2 , l_3)(l_4 , l_5))} \cdot [l_2 , l_3] \cdot [l_4 , l_5] = \dfrac{l_1 l_2 l_3 l_4 l_5}{\big( (l_2 , l_3) (l_4 , l_5) , l_1 (l_2 , l_3 , l_4 , l_5) \big)}.
\end{array}
\end{displaymath}
\end{proof}

In this section, we first prove Lemma~\ref{mahlem} without any condition. Then we give the proof of the Conjecture~\ref{conjec} for any GCD domain. In order to do this, we need some lemmas.

\begin{lem} Let $D_{3,3}$ be as in Figure~\ref{d33} and let $\{ F_1 , F_2 ,F_3 , F_4 \} \subset R_{(D_{3,3} , \alpha)}$. Then the products $l_1 l_2 l_4$ , $l_1 l_2 l_5$ , $l_1 l_3 l_4$ , $l_1 l_3 l_5$ , $l_2 l_3 l_4$ , $l_2 l_3 l_5$ , $l_2 l_4 l_5$ and $l_3 l_4 l_5$ divide $\big\vert F_1 \text{ } F_2 \text{ } F_3 \text{ } F_4 \big\vert$.
\label{33lem1}
\end{lem}

\begin{proof} Since the corresponding edges to $l_2 , l_3 , l_4$ and $l_5$ form an outer cycle on $(D_{3,3} , \alpha)$, we conclude that $l_2 l_3 l_4$ , $l_2 l_3 l_5$ , $l_2 l_4 l_5$ and $l_3 l_4 l_5$ divides $\big\vert F_1 \text{ } F_2 \text{ } F_3 \text{ } F_4 \big\vert$ by Lemma~\ref{cyclem}. In order to see that $l_1 l_2 l_5$ divides $\big\vert F_1 \text{ } F_2 \text{ } F_3 \text{ } F_4 \big\vert$, where $F_i = (f_{i1} , f_{i2} , f_{i3} , f_{i4}) \in R_{(D_{3,3} , \alpha)}$ for $i = 1,2,3,4$, we consider the determinant
\begin{displaymath}
\big\vert F_1 \text{ } F_2 \text{ } F_3 \text{ } F_4 \big\vert = \begin{vmatrix} f_{14} & f_{24} & f_{34} & f_{44} \\ f_{13} & f_{23} & f_{33} & f_{43} \\ f_{12} & f_{22} & f_{32} & f_{42} \\ f_{11} & f_{21} & f_{31} & f_{41} \\ \end{vmatrix}.
\end{displaymath} By some suitable row operations on the determinant, we obtain 

\begin{displaymath}
\begin{array}{ccl}
\big\vert F_1 \text{ } F_2 \text{ } F_3 \text{ } F_4 \big\vert &=& \begin{vmatrix} f_{14} & f_{24} & f_{34} & f_{44} \\ f_{13} & f_{23} & f_{33} & f_{43} \\ f_{12} & f_{22} & f_{32} & f_{42} \\ f_{11} & f_{21} & f_{31} & f_{41} \\ \end{vmatrix} = \begin{vmatrix} f_{14} - f_{11} & f_{24} - f_{21} & f_{34} - f_{31} & f_{44} - f_{41} \\ f_{13} -f_{12} & f_{23} - f_{22} & f_{33} - f_{32} & f_{43} - f_{42} \\ f_{12} & f_{22} & f_{32} & f_{42} \\ f_{11} - f_{12} & f_{21} - f_{22} & f_{31} - f_{32} & f_{41} - f_{42} \end{vmatrix} \vspace{.3cm} \\ 
&=& \begin{vmatrix} x_{14} l_5 & x_{24} l_5 & x_{34} l_5 & x_{44} l_5 \\ x_{13} l_2 & x_{23} l_2 & x_{33} l_2 & x_{43} l_2 \\ f_{12} & f_{22} & f_{32} & f_{42} \\ x_{11} l_1 & x_{21} l_1 & x_{31} l_1 & x_{41} l_1 \end{vmatrix} = l_1 l_2 l_5 \underbrace{\begin{vmatrix} x_{14} & x_{24} & x_{34} & x_{44} \\ x_{13} & x_{23} & x_{33} & x_{43} \\ f_{12} & f_{22} & f_{32} & f_{42} \\ x_{11} & x_{21} & x_{31} & x_{41} \end{vmatrix}}_{\in R}.
\end{array}
\end{displaymath} for some $x_{ij}$.
Hence we see that $l_1 l_2 l_5$ divides $\big\vert F_1 \text{ } F_2 \text{ } F_3 \text{ } F_4 \big\vert$. Similarly, one can also show that the other products divide the determinant.
\end{proof}

As we see in Lemma~\ref{33lem1}, the products of three edge labels whose corresponding edges do not form a subcycle in $D_{3,3}$ divide the determinant $\big\vert F_1 \text{ } F_2 \text{ } F_3 \text{ } F_4 \big\vert$. We generalize the statement of this lemma to diamond graph $D_{m,n}$ later (see Lemma~\ref{dmnlem}). The following statement is a corollary of Lemma~\ref{33lem1}:

\begin{cor} Let $D_{3,3}$ be as in Figure~\ref{d33} and let $\{ F_1 , F_2 ,F_3 , F_4 \} \subset R_{(D_{3,3} , \alpha)}$. Then
\begin{displaymath}
\Big[ l_1 l_2 l_4 , l_1 l_2 l_5 , l_1 l_3 l_4 , l_1 l_3 l_5 , l_2 l_3 l_4 , l_2 l_3 l_5 , l_2 l_4 l_5 , l_3 l_4 l_5 \Big] = \left[ l_1 \cdot [l_2 , l_3] \cdot [l_4 , l_5] \text{ , } \dfrac{l_2 l_3 l_4 l_5}{(l_2 , l_3 , l_4 , l_5)} \right]
\end{displaymath}
divides $\big\vert F_1 \text{ } F_2 \text{ } F_3 \text{ } F_4 \big\vert$.
\end{cor}

The following lemma shows that we can consider $Q_{D_{3,3}}$ as the least common multiple of the products that defined above:

\begin{lem}
$Q_{D_{3,3}} = \left[ l_1 \cdot [l_2 , l_3] \cdot [l_4 , l_5] \text{ , } \dfrac{l_2 l_3 l_4 l_5}{(l_2 , l_3 , l_4 , l_5)} \right]$.
\label{33lem2}
\end{lem}

\begin{proof}
\begin{displaymath}
\begin{array}{ccl}
\left[ l_1 \cdot [l_2 , l_3] \cdot [l_4 , l_5] \text{ , } \dfrac{l_2 l_3 l_4 l_5}{(l_2 , l_3 , l_4 , l_5)} \right] &=& \left[ l_1 \dfrac{l_2 l_3}{(l_2 , l_3)} \dfrac{l_4 l_5}{(l_4 , l_5)} \text{ , } \dfrac{l_2 l_3 l_4 l_5}{(l_2 , l_3 , l_4 , l_5)} \right]\vspace{.3cm} \\ &=& \dfrac{\dfrac{l_1 l_2 l_3 l_4 l_5}{(l_2 , l_3)(l_4 , l_5)} \cdot \dfrac{l_2 l_3 l_4 l_5}{(l_2 , l_3 , l_4 , l_5)}}{\left( \dfrac{l_1 l_2 l_3 l_4 l_5}{(l_2 , l_3)(l_4 , l_5)} , \dfrac{l_2 l_3 l_4 l_5}{(l_2 , l_3 , l_4 , l_5)} \right)}\vspace{.3cm} \\ &=&  \dfrac{l_1 l_2 l_3 l_4 l_5 \cdot l_2 l_3 l_4 l_5}{(l_2 , l_3)(l_4 , l_5)(l_2 , l_3 , l_4 , l_5) \left( \tfrac{l_1 l_2 l_3 l_4 l_5}{(l_2 , l_3)(l_4 , l_5)} , \tfrac{l_2 l_3 l_4 l_5}{(l_2 , l_3 , l_4 , l_5)} \right)} \vspace{.3cm} \\ &=& \dfrac{l_1 l_2 l_3 l_4 l_5 \cdot l_2 l_3 l_4 l_5}{\Big( l_1 l_2 l_3 l_4 l_5 (l_2 , l_3 , l_4 , l_5) \text{ , } l_2 l_3 l_4 l_5 (l_2 , l_3)(l_4 , l_5) \Big)}\vspace{.3cm} \\ &=& \dfrac{l_1 l_2 l_3 l_4 l_5}{\big( (l_2 , l_3) (l_4 , l_5) , l_1 (l_2 , l_3 , l_4 , l_5) \big)} = Q_{D_{3,3}}.
\end{array}
\end{displaymath}
\end{proof}

We state Lemma~\ref{mahlem} without any condition for any GCD domain as a corollary below.  Together with Lemma~\ref{33lem1} and ~\ref{33lem2}, the proof follows easily.
\begin{cor} Let $(D_{3,3} , \alpha)$ be an edge labeled diamond graph as in Figure~\ref{d33} and let $\{ F_1 , F_2 ,F_3 , F_4 \} \subset R_{(D_{3,3} , \alpha)}$. Then $Q_{D_{3,3}}$ divides $\big\vert F_1 \text{ } F_2 \text{ } F_3 \text{ } F_4 \big\vert$.
\label{d33cor}
\end{cor}

We start to prove Conjecture~\ref{conjec} for any GCD domain.

\begin{lem} Let $(D_{3,3} , \alpha)$ be an edge labeled diamond graph as in Figure~\ref{d33}. Let $\{ F_1 , F_2 , F_3 , F_4 \} \subset R_{(D_{3,3} , \alpha)}$. If $\big\vert F_1 \text{ } F_2 \text{ } F_3 \text{ } F_4 \big\vert = r \cdot Q_{D_{3,3}}$ where $r \in R$ is a unit, then $\{ F_1 , F_2 , F_3 , F_4 \}$ forms a basis for $R_{(D_{3,3} , \alpha)}$.
\label{cycd331}
\end{lem}

\begin{proof} Since $\big\vert F_1 \text{ } F_2 \text{ } F_3 \text{ } F_4 \big\vert = r \cdot Q_{D_{3,3}}$, the set $\{ F_1 , F_2 , F_3 , F_4 \}$ is linearly independent. First we claim that $Q_{D_{3,3}} \cdot R^4 \in \big\langle F_1 , F_2 , F_3 , F_4 \big\rangle$. In order to see this, let $\big( r_1 , r_2 , r_3 , r_4 \big) \in R^4$. We need to show the existence of $a_1 , a_2 , a_3 , a_4 \in R$ such that
\begin{displaymath}
\big( Q_{D_{3,3}} r_1 \text{ , } Q_{D_{3,3}} r_2 \text{ , } Q_{D_{3,3}} r_3 \text{ , } Q_{D_{3,3}} r_4 \big) = \sum\limits_{i = 1} ^4 a_i F_i.
\end{displaymath}
Rewrite this equality in a matrix form as
\begin{displaymath}
\begin{bmatrix} Q_{D_{3,3}} r_4 \\ \vdots \\ Q_{D_{3,3}} r_1 \end{bmatrix} = \begin{bmatrix} f_{14} & \ldots & f_{44} \\ \vdots \\ f_{11} & \ldots & f_{41} \end{bmatrix} \begin{bmatrix} a_4 \\ \vdots \\ a_1 \end{bmatrix}.
\end{displaymath}
By Cramer's rule we get
\begin{displaymath}
a_1 = \dfrac{\begin{vmatrix} Q_{D_{3,3}} r_4 & \ldots & f_{44} \\ \vdots \\ Q_{D_{3,3}} r_1 & \ldots & f_{41} \end{vmatrix}}{Q_{D_{3,3}}} = \begin{vmatrix} r_4 & \ldots & f_{44} \\ \vdots \\ r_1 & \ldots & f_{41} \end{vmatrix} \in R.
\end{displaymath}
We can see the existence of $a_2 , a_3$ and $a_4$ by the same way. Hence we conclude that $Q_{D_{3,3}} R^4 \in \big\langle F_1 , F_2 , F_3 , F_4 \big\rangle$. Let $F \in R_{(D_{3,3} , \alpha)}\subset R^4$. Then $Q_{D_{3,3}} F \in \big\langle F_1 , F_2 , F_3 , F_4 \big\rangle$ and so 
\begin{displaymath}
Q_{D_{3,3}} F = \sum\limits_{i = 1} ^4 r_i F_i
\end{displaymath}
for some $r_i \in R$. Then
\begin{displaymath}
\begin{array}{ccl}
r_i r \cdot Q_{D_{3,3}} = r_i \big\vert F_1 \text{ } F_2 \text{ } F_3 \text{ } F_4 \big\vert = \big\vert F_1 \text{ } \ldots \text{ } r_i F_i \text{ } \ldots \text{ } F_4 \big\vert &=& \big\vert F_1 \text{ } \ldots \text{ } \sum\limits_{j = 1} ^4 r_j F_j \text{ } \ldots \text{ } F_4 \big\vert \vspace{.2cm} \\ &=& \big\vert F_1 \text{ } \ldots \text{ } Q_{D_{3,3}} F \text{ } \ldots \text{ } F_4 \big\vert \vspace{.2cm} \\ &=& Q_{D_{3,3}} \underbrace{\big\vert F_1 \text{ } \ldots \text{ } F \text{ } \ldots \text{ } F_4 \big\vert}_{\in Q_{D_{3,3}}R} \in Q_{D_{3,3}} ^2 R.
\end{array}
\end{displaymath}
Since $r$ is a unit by assumption, $Q_{D_{3,3}}$ divides $r_i$ and
\begin{displaymath}
F = \sum\limits_{i = 1} ^4 \bigg( \dfrac{r_i}{Q_{D_{3,3}}} \bigg) F_i \in \big\langle F_1 , F_2 , F_3 , F_4 \big\rangle.
\end{displaymath}
\end{proof}

The following theorem is proved for integer splines in~\cite{Mah} by using flow-up bases. We use a different approach to prove the statement since the existence of flow-up bases is not guaranteed over GCD domains.

\begin{theo} Let $(D_{3,3} , \alpha)$ be an edge labeled diamond graph as in Figure~\ref{d33}. Let $\{ F_1 , F_2 , F_3 , F_4 \} \subset R_{(D_{3,3} , \alpha)}$. If $\{ F_1 , F_2 , F_3 , F_4 \}$ is a basis for $R_{(D_{3,3} , \alpha)}$ , then $\big\vert F_1 \text{ } F_2 \text{ } F_3 \text{ } F_4 \big\vert = r \cdot Q_{D_{3,3}}$ where $r \in R$ is a unit.
\label{d33theo}
\end{theo}

\begin{proof} Since $\{ F_1 , F_2 , F_3 , F_4 \} \subset R_{(D_{3,3} , \alpha)}$, the determinant $\big\vert F_1 \text{ } F_2 \text{ } F_3 \text{ } F_4 \big\vert = r \cdot Q_{D_{3,3}}$ for some $r \in R$ by Corollary~\ref{d33cor}. We will show that $r$ is a unit. Let $d_1 = (l_2 , l_3)$ and $d_2 = (l_4 , l_5)$. Then we have $l_2 = d_1 l' _2$, $l_3 = d_1 l' _3$ with $(l' _2 , l' _3) = 1$ and $l_4 = d_2 l' _4$, $l_5 = d_2 l' _5$ with $(l' _4 , l' _5) = 1$. Consider the following matrices:
\begin{displaymath}
A_1 = \begin{bmatrix}
1 & 0 & 0 & [l_4 , l_5] \\
1 & 0 & [l_2 , l_3] & 0 \\
1 & 0 & 0 & 0 \\
1 & \big[ l_1 , (l_2 , l_3) , (l_4 , l_5) \big] l' _3 l' _5 & 0 & 0
\end{bmatrix} , \hspace{.2cm}
A_2 = \begin{bmatrix}
1 & \big[ l_1 , (l_2 , l_3) , (l_4 , l_5) \big] l' _4 l' _2 & 0 & [l_4 , l_5] \\
1 & \big[ l_1 , (l_2 , l_3) , (l_4 , l_5) \big] l' _4 l' _2 & [l_2 , l_3] & 0 \\
1 & 0 & 0 & 0 \\
1 & \big[ l_1 , (l_2 , l_3) , (l_4 , l_5) \big] l' _4 l' _2 & 0 & 0
\end{bmatrix},
\end{displaymath}

\begin{displaymath}
A_3 = \begin{bmatrix}
1 & \big[ l_1 , (l_2 , l_3) , (l_4 , l_5) \big] l' _3 l' _4 & 0 & [l_4 , l_5] \\
1 & 0 & [l_2 , l_3] & 0 \\
1 & 0 & 0 & 0 \\
1 & \big[ l_1 , (l_2 , l_3) , (l_4 , l_5) \big] l' _3 l' _4 & 0 & 0
\end{bmatrix} , \hspace{.2cm}
A_4 = \begin{bmatrix}
1 & 0 & 0 & [l_4 , l_5] \\
1 & \big[ l_1 , (l_2 , l_3) , (l_4 , l_5) \big] l' _2 l' _5 & [l_2 , l_3] & 0 \\
1 & 0 & 0 & 0 \\
1 & \big[ l_1 , (l_2 , l_3) , (l_4 , l_5) \big] l' _2 l' _5 & 0 & 0
\end{bmatrix}.
\end{displaymath}
It can be easily seen that each column of $A_1 , A_2 , A_3$ and $A_4$ is an element of $R_{(D_{3,3} , \alpha)}$. By Proposition~\ref{detprop2}, $\big\vert F_1 \text{ } F_2 \text{ } F_3 \text{ } F_4 \big\vert = r \cdot Q_{D_{3,3}}$ divides $\big\vert A_1 \big\vert = Q_{D_{3,3}} \cdot l' _3 l' _5$. Hence $r$ divides $l' _3 l' _5$. One can conclude that $r$ divides $l' _2 l' _4 , l' _3 l' _4$ and $l' _2 l' _5$ by the same observation. Thus we have
\begin{displaymath}
r \Big\vert \big( l' _2 l' _4 , l' _2 l' _5 , l' _3 l' _4 , l' _3 l' _5 \big) = \big( (l' _2 l' _4 , l' _2 l' _5) , (l' _3 l' _4 , l' _3 l' _5) \big) = \big( l' _2 \underbrace{(l' _4 , l' _5)}_{=1} , l' _3 \underbrace{(l' _4 , l' _5)}_{=1} \big) = (l' _2 , l' _3) = 1
\end{displaymath}
and so $r$ is a unit.
\end{proof}

We have the following result as the main theorem of this section by combining Lemma~\ref{cycd331} and Theorem~\ref{d33theo}. This result was given as a conjecture over the base ring $\mathbb{Z}$ in~\cite{Mah}.

\begin{theo} Let $(D_{3,3} , \alpha)$ be an edge labeled diamond graph as in Figure~\ref{d33}. Then $\{ F_1 , F_2 , F_3 , F_4 \} \subset R_{(D_{3,3} , \alpha)}$ is a basis for $R_{(D_{3,3} , \alpha)}$ if and only if $\big\vert F_1 \text{ } F_2 \text{ } F_3 \text{ } F_4 \big\vert = r \cdot Q_{D_{3,3}}$ where $r \in R$ is a unit.
\label{d33theo2}
\end{theo}

\subsection{\normalsize{Determinant of Splines on Diamond Graph $D_{m,n}$}} \label{secdmn}

In this section, we generalize the results from Section~\ref{secd33} to the diamond graph $D_{m,n}$. We use the following illustration of $D_{m,n}$:

\begin{figure}[H]
\begin{center}
\scalebox{0.16}{\includegraphics{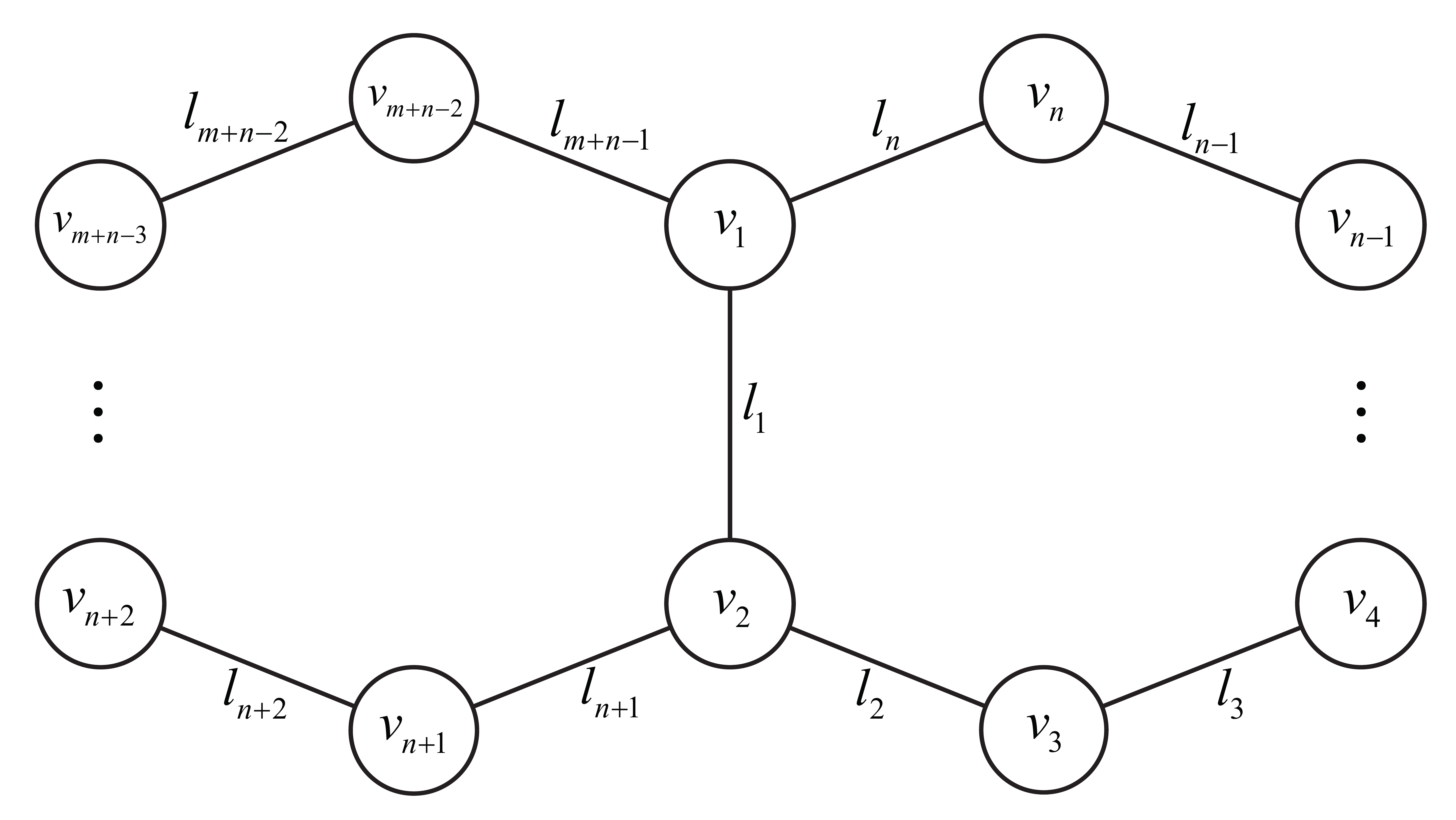}}
\caption{Edge labeled diamond graph $(D_{m,n} , \alpha)$}
\label{dmn}
\end{center}
\end{figure}

We can write $Q_{D_{m,n}}$ in terms of edge labels.
\begin{lem} Let $(D_{m,n} , \alpha)$ be as in Figure~\ref{dmn}. Then
\begin{displaymath}
Q_{D_{m,n}} = \dfrac{l_1 l_2 \cdots l_{m+n-1}}{\big( (l_2 , \ldots , l_n) (l_{n+1} , \ldots , l_{m+n-1}) \text{ , } l_1 (l_2 , \ldots , l_{m+n-1}) \big)}.
\end{displaymath}
\end{lem}

\begin{proof}
\begin{displaymath}
\begin{split}
Q_{D_{m,n}} =& \big[ l_1 , (l_2 , \ldots , l_n) , (l_{n+1} , \ldots , l_{m+n-1})\big] \cdot \big[ l_2 , (l_3 , \ldots , l_n) \big] \cdots \big[ l_{n-2} , (l_{n-1} , l_n) \big] \cdot \\ \cdot& \big[ l_{n-1} , l_n \big] \cdot \big[ l_{n+1}, (l_{n+2} , \ldots , l_{m+n-1}) \big] \cdots \big[ l_{m+n-3} , (l_{m+n-2} , l_{m+n-1}) \big] \cdot \\ \cdot& \big[ l_{m+n-2} , l_{m+n-1} \big] \\
=& \dfrac{l_1 l_2 \cdots l_{m+n-1}}{\big( (l_2 , \ldots , l_n) (l_{n+1} , \ldots , l_{m+n-1}) \text{ , } l_1 (l_2 , \ldots , l_{m+n-1}) \big)}.
\end{split}
\end{displaymath}
The last equality can be shown easily by rewriting the least common multiples and greatest common divisors explicitly.
\end{proof}

Let $\{ F_1 , \ldots , F_{m+n-2} \} \subset R_{(D_{m,n} , \alpha)}$. We will prove that $Q_{D_{m,n}}$ divides $\big\vert F_1 \text{ } F_2 \text{ } \ldots \text{ } F_{m+n-2} \big\vert$. In order to do this, first we determine the products of certain edge labels that divide the determinant $\big\vert F_1 \text{ } F_2 \text{ } \ldots \text{ } F_{m+n-2} \big\vert$. We claim that the product of the edge labels $l' _1 , \ldots , l' _{m+n-3}$ whose corresponding edges do not contain the subcycles $C_m$ or $C_n$ of $D_{m,n}$ divide the determinant. We can characterize such edge labels in two forms:
\begin{itemize}
	\item $p_i = l_2 l_3 \cdots \hat{l_i} \cdots l_n$ where $2 \leq i \leq m+n-1$,
	\item $q_{j,k} = (l_1 \cdots \hat{l_j} \cdots l_n)( l_{n+1} \cdots \hat{l_k} \cdots l_{m+n-1})$ where $2 \leq j \leq n$ and $n+1 \leq k \leq m+n-1$.
\end{itemize}

\begin{lem} Let $D_{m,n}$ be as in Figure~\ref{dmn} and let $\{ F_1 , \ldots , F_{m+n-2} \} \subset R_{(D_{m,n} , \alpha)}$. Then $p_i$ and $q_{j,k}$ divide the determinant $\big\vert F_1 \text{ } F_2 \text{ } \ldots \text{ } F_{m+n-2} \big\vert$ for all $2 \leq i \leq m+n-1$, $2 \leq j \leq n$ and $n+1 \leq k \leq m+n-1$.
\label{dmnlem}
\end{lem}

\begin{proof} Since $l_2 , \ldots , l_{m+n-1}$ correspond to the edge labels of the outer cycle $C_{m+n}$, $p_i$ divides the determinant $\big\vert F_1 \text{ } F_2 \text{ } \ldots \text{ } F_{m+n-2} \big\vert$ for all $2 \leq i \leq n$ by Lemma~\ref{cyclem}. For $q_{j,k}$, we consider two subcylces $C_m$ and $C_n$ of $D_{m,n}$. Since $l_1 , l_{n+1} , \ldots , l_{m+n-1}$ and $l_1 , l_2 , \ldots , l_n$ correspond to the edge labels of the subcycles $C_m$ (left hand side) and $C_n$ (right hand side) respectively, we can do first suitable row operations on the matrix $[F_1 \text{ } F_2 \text{ } \ldots \text{ } F_{m+n-2}]$ without changing last two rows so  that $l_{n+1} \cdots \hat{l_k} \cdots l_{m+n-1}$ divides the determinant and then we carry on doing necessary row operations on the last $n$ rows of the changed matrix so that $l_1 \cdots \hat{l_j} \cdots l_n$  divides the determinant as in the proof of Lemma~\ref{cyclem}. Notice that the necessary row operations for $l_1 \cdots \hat{l_j} \cdots l_n$ and $l_{n+1} \cdots \hat{l_k} \cdots l_{m+n-1}$ are independent. Hence we conclude that $q_{j,k}$ divides the determinant.
\end{proof}

Define the following two sets:
\begin{itemize}
	\item $\mathscr{P} = \{ p_i \text{ $\vert$ } 2 \leq i \leq m+n-1 \}$,
	\item $\mathscr{Q} = \{ q_{j,k} \text{ $\vert$ } 2 \leq j \leq n, n+1 \leq k \leq m+n-1 \}$
\end{itemize}
Since each element of $\mathscr{P} \cup \mathscr{Q}$ divides $\big\vert F_1 \text{ } F_2 \text{ } \ldots \text{ } F_{m+n-2} \big\vert$ by Lemma~\ref{dmnlem}, the least common multiple $\big[ \mathscr{P} \cup \mathscr{Q} \big]$ also divides the determinant. The following lemma shows that we can consider $Q_{D_{m,n}}$ as the least common multiple of the elements of $\mathscr{P} \cup \mathscr{Q}$:

\begin{lem} $Q_{D_{m,n}} = \big[ \mathscr{P} \cup \mathscr{Q} \big]$.
\label{dmnlem2}
\end{lem}

\begin{proof}
\begin{flalign} \nonumber
\big[ \mathscr{P} \cup \mathscr{Q} \big] &= \Big[ \big[ l_2 \cdots \hat{l_i} \cdots l_{m+n-1} \text{ $\vert$ } 2 \leq i \leq m+n-1 \big] \text{ , } \big[ l_1 \cdots \hat{l_j} \cdots l_n l_{n+1} \cdots \hat{l_k} \cdots l_{m+n-1}\text{ $\vert$ } \substack{2 \leq j \leq n, \\ n+1 \leq k \leq m+n-1} \big] \Big]& \vspace{.3cm} \\ \nonumber
&= \left[ \dfrac{l_2 \cdots l_{m+n-1}}{(l_2, \ldots ,l_{m+n-1})} \text{ , } l_1 \cdot \dfrac{l_2 \cdots l_n}{(l_2 , \ldots , l_n)} \cdot \dfrac{l_{n+1} \cdots l_{m+n-1}}{(l_{n+1}, \ldots ,l_{m+n-1})} \right]\vspace{.3cm} \\ \nonumber
&= \dfrac{ \dfrac{l_1 \cdots l_{m+n-1} \cdot l_2 \cdots l_{m+n-1}}{(l_2 , \ldots , l_n)(l_{n+1}, \ldots ,l_{m+n-1})(l_2, \ldots ,l_{m+n-1})} }{\left( l_1 \cdot \dfrac{l_2 \cdots l_n}{(l_2 , \ldots , l_n)} \cdot \dfrac{l_{n+1} \cdots l_{m+n-1}}{(l_{n+1}, \ldots ,l_{m+n-1})} \text{ , } \dfrac{l_2 \cdots l_{m+n-1}}{(l_2, \ldots ,l_{m+n-1})} \right)}\vspace{.3cm} \\ \nonumber
&= \dfrac{l_1 \cdots l_{m+n-1} \cdot l_2 \cdots l_{m+n-1}}{\big( l_1 \cdots l_{m+n-1}(l_2, \ldots ,l_{m+n-1}) \text{ , } l_2 \cdots l_{m+n-1} (l_2 , \ldots , l_n) (l_{n+1}, \ldots ,l_{m+n-1}) \big) }\vspace{.3cm} \\ \nonumber
&= \dfrac{l_1 \cdots l_{m+n-1}}{\big( l_1 (l_2, \ldots ,l_{m+n-1}) \text{ , } (l_2 , \ldots , l_n) (l_{n+1}, \ldots ,l_{m+n-1}) \big) }\vspace{.3cm} \\ \nonumber
&= Q_{D_{m,n}}.
\end{flalign}
\end{proof}

Lemma~\ref{dmnlem2} gives us the following result:

\begin{cor} Let $(D_{m,n} , \alpha)$ be as in Figure~\ref{dmn} and let $\{ F_1 , \ldots , F_{m+n-2} \} \subset R_{(D_{m,n} , \alpha)}$. Then $Q_{D_{m,n}}$ divides $\big\vert F_1 \text{ } F_2 \text{ } \ldots \text{ } F_{m+n-2} \big\vert$.
\label{dmncor2}
\end{cor}

We give the final result of this section below.

\begin{theo} Let $(D_{m,n} , \alpha)$ be as in Figure~\ref{dmn}. Then $\{ F_1 , \ldots , F_{m+n-2} \} \subset R_{(D_{m,n} , \alpha)}$ forms a basis for $R_{(D_{m,n} , \alpha)}$ if and only if $\big\vert F_1 \text{ } F_2 \text{ } \ldots \text{ } F_{m+n-2} \big\vert = r \cdot Q_{D_{m,n}}$ where $r \in R$ is a unit.
\label{dmntheo}
\end{theo}

\begin{proof} The second part of the theorem can be shown by similar techniques in the proof of Lemma~\ref{cycd331}. For the proof of the first part, we assume that  $\{ F_1 , \ldots , F_{m+n-2} \}$ forms a basis for $R_{(D_{m,n} , \alpha)}$. By Corollary~\ref{dmncor2}, $Q_{D_{m,n}}$ divides $\big\vert F_1 \text{ } F_2 \text{ } \ldots \text{ } F_{m+n-2} \big\vert$, say $\big\vert F_1 \text{ } F_2 \text{ } \ldots \text{ } F_{m+n-2} \big\vert = r \cdot Q_{D_{m,n}}$ for some $r \in R$. We will show that $r$ is a unit.

First fix the following notations:
\begin{displaymath}
d_i = 
\begin{cases}
(l_{i+1} , \ldots , l_n), & \text{ for } i = 1,2,\ldots,n-2 \\
(l_{i+1} , \ldots , l_{m+n-1}), & \text{ for } i = n,\ldots,m+n-3.
\end{cases}
\end{displaymath}
Here for any $i \in \{ 1,\ldots,n-2\}$, we say $l_j = d_i l^{(i)} _j$ if $l_j \in \{ l_{i+1} , \ldots , l_n \}$. Similarly we say $l_j = d_i l^{(i)} _j$ if $l_j \in \{ l_{i+1} , \ldots , l_{m+n-1} \}$ for any $i \in \{ n,\ldots,m+n-3\}$. For a fixed $i \in \{ 1,\ldots,n-2\}$, it is obvious that $\big( l^{(i)} _{i+1} , \ldots , l^{(i)} _n \big) = 1$. Similarly $\big( l^{(i)} _{i+1} , \ldots , l^{(i)} _{m+n-1} \big) = 1$ for a fixed $i \in \{ n,\ldots,m+n-3\}$. 

Fix the following sets:
\begin{displaymath}
B_i =
\begin{cases} 
\left\{ l^{(i)} _j \text{ $\vert$ } j = i+1 , \ldots , n \right\}, & i = 1, \ldots , n-2 \\
\left\{ l^{(i)} _j \text{ $\vert$ } j = i+1 , \ldots , m+n-1 \right\}, & i = n , \ldots , m+n-3.
\end{cases}
\end{displaymath}
Notice that the greatest common divisor of all elements of each $B_i$ is equal to $1$. We will show that $r$ divides any product $b_1 \cdots b_{n-2} \cdot b_n \cdots b_{m+n-3}$ where $b_i \in B_i$. In order to do this, we construct sets $\mathbb{A}^t$ of column matrices for $t = 0,\ldots , n-1 , n+1 , \ldots , m+n-2$ with entries in order from bottom to top as follows. 

\begin{itemize}
\item For $t=0$, fix the following set:
\begin{displaymath}
\mathbb{A}^0 = \left\{ \begin{bmatrix} 1 \\ \vdots \\ 1 \end{bmatrix} \right\}.
\end{displaymath}
\item For $t = 1$, construct the following column matrices $A_{i,j}$ for $2 \leq i \leq n$ and $n+1 \leq j \leq m+n-1$
\begin{displaymath}
\big[ A_{i,j} \big]_k = 
\begin{cases} [l_1 , d_1 , d_n] l^{(1)} _i l^{(n)} _j, & 1 < k \leq i \text{ or } n < k \leq j-1 \\
0, & \text{otherwise} \end{cases}
\end{displaymath}
and fix the following set:
\begin{displaymath}
\mathbb{A}^1 = \left\{ A_{i,j} \text{ $\vert$ } 2 \leq i \leq n \text{ , } n+1 \leq j \leq m+n-2 \right\}.
\end{displaymath}
\item For a fixed $t$ with $2 \leq t \leq n-2$, define the entries of the columns $A^t _j$ for $j = t+1, \ldots , n$ as follows
\begin{displaymath}
\big[ A^t _j \big]_k = 
\begin{cases} [l_t , d_t] l^{(t)} _j, & t < k \leq j\\
0, & \text{otherwise} \end{cases}
\end{displaymath}
and fix the following set:
\begin{displaymath}
\mathbb{A}^t = \left\{ A^t _j \text{ $\vert$ } j = 3, \ldots , n \right\}.
\end{displaymath}
\item For $t = n-1$, fix the following set:
\begin{displaymath}
\mathbb{A}^{n-1} = \left\{ \begin{gmatrix}[b]
 0 \\ \vdots \\ 0 \\ [ l_{n-1} , l_n ] \\ 0 \\ \vdots \\ 0
\rowops
\mult{3}{\leftarrow\text{$n$-th row}}
\end{gmatrix} \right\}.
\end{displaymath}
\item For a fixed $t$ with $n+1 \leq t \leq m+n-3$, define the entries of the columns $A^t _j$ for $j = t+1, \ldots , m+n-1$ as follows
\begin{displaymath}
\big[ A^t _j \big]_k = 
\begin{cases} [l_t , d_t] l^{(t)} _j, & t \leq k \leq j-1\\
0, & \text{otherwise} \end{cases}
\end{displaymath}
and fix the following set:
\begin{displaymath}
\mathbb{A}^t = \left\{ A^t _j \text{ $\vert$ } j = n+1, \ldots , m+n-1 \right\}.
\end{displaymath}
\item For $t = m+n-2$, fix the following set:
\begin{displaymath}
\mathbb{A}^{m+n-2} = \left\{ \begin{bmatrix}
 [l_{m+n-2} , l_{m+n-1}] \\ 0 \\ \vdots \\ 0
\end{bmatrix} \right\}.
\end{displaymath}
\end{itemize}

\noindent
It is easy to see that each element of the sets $\mathbb{A}^0 , \mathbb{A}^1 , \ldots , \mathbb{A}^{n-1} , \mathbb{A}^{n+1} , \ldots , \mathbb{A}^{m+n-2}$ is also an element of $R_{(D_{m,n} , \alpha)}$. Consider the matrix $A$ whose columns come from these sets as below.
\begin{displaymath}
A = 
\begin{bmatrix}
A_0 & A_1 & \ldots & A_{n-1} & A_{n+1} & \ldots & A_{m+n-2}
\end{bmatrix}
\end{displaymath}
where $A_i \in \mathbb{A}^i$. Notice that $A$ is an upper triangular matrix and for each element $b_1 \cdots b_{n-2} \cdot b_n \cdots b_{m+n-3}$ where $b_i \in B_i$, we can construct $A$ such that 
\begin{displaymath}
\big\vert A \big\vert = b_1 \cdots b_{n-2} \cdot b_n \cdots b_{m+n-3} \cdot Q_{D_{m,n}}
\end{displaymath}
by choosing suitable elements from the sets $\mathbb{A}^i$. By Proposition~\ref{detprop2}, $r \cdot Q_{D_{m,n}} = \big\vert F_1 \text{ } F_2 \text{ } \ldots \text{ } F_{m+n-2} \big\vert$ divides $\big\vert A \big\vert$ and we conclude that $r$ divides all elements of the set $B_1 \cdots B_{n-2} \cdot B_n \cdots B_{m+n-3}$. Hence
\begin{displaymath}
r \text{ $\Big\vert$ }  \big( B_1 \cdots B_{n-2} \cdot B_n \cdots B_{m+n-3} \big) = \big( B_1 \big) \cdots \big( B_{n-2} \big) \cdot \big( B_n \big) \cdots \big( B_{m+n-3} \big) \\ = 1.
\end{displaymath}
and $r$ is a unit.
\end{proof}

We also believe that the statement of Theorem~\ref{dmntheo} can be easily extended to the graphs that consist of $n$ cycles sharing a common edge or a common tree similarly in our work.

\subsection{Determinant of Splines on Trees}

In this section we give basis criteria for spline modules on trees by using determinant. We already know that spline modules on trees have a free module structure. Moreover, they have flow-up bases. So we can use flow-up bases to give basis criteria. We formularize $Q_G$ for trees as follows. 

\begin{lem} Let $G$ be a tree with $n$ vertices and $k$ edges. Then
\begin{displaymath}
Q_G = l_1 \cdots l_k.
\end{displaymath}
\end{lem}

\begin{proof} Fix a vertex $v_i$ and consider the following zero trial $\textbf{p}^{(i,0)}$ which is not an edge:

\begin{figure}[H]
\begin{center}
\scalebox{0.16}{\includegraphics{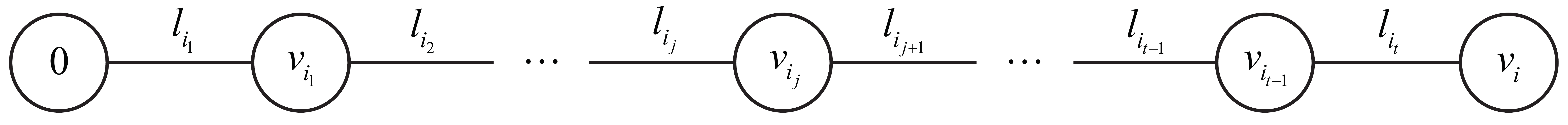}}
\caption{Zero trial of $v_i$}
\label{path}
\end{center}
\end{figure}

\noindent
Here notice that $i \leq i_1 , \ldots , i_{t-1}$. We will show that the greatest common divisor of such zero trials cancels in the product 
\begin{displaymath}
Q_G = \prod\limits_{i=2} ^k \big[ \big\{ \big( \textbf{p}_t ^{(i,0)} \big) \text{ $\big\vert$ } t = 1,\ldots, m_i \big\} \big]
\end{displaymath}
where $m_i$ is the number of the zero trials of $v_i$. In order to see this, let $i_j = \text{ min } \{ i_1 , \ldots , i_{t-1} \}$. Then $l_{i_1} \ldots l_{i_j}$ and $l_{i_{j+1}} \ldots l_{i_{t}}$ are two zero trials of $v_{i_j}$ and we have
\begin{displaymath}
\begin{array}{ccl}
Q_G &=& \Big[ \big( l_{i_1} \ldots l_{i_t} \big) , \big[ \big\{ \big( \text{other zero trials of } v_i \big) \big\} \big] \Big] \cdot \vspace{.2cm} \\
&& \cdot \Big[ \big[ \big(l_{i_1} \ldots l_{i_j} \big) , \big( l_{i_{j+1}} \ldots l_{i_{t}} \big) \big] , \big[ \big\{ \big( \text{other zero trials of } v_{i_j} \big) \big\} \big] \Big] \cdot \text{ (other least common multiples) } \vspace{.2cm} \\
&=& \big( l_{i_1} \ldots l_{i_t} \big) \cdot \text{ (other terms) } \cdot \big[ \big(l_{i_1} \ldots l_{i_j} \big) , \big( l_{i_{j+1}} \ldots l_{i_{t}} \big) \big] \cdot \text{ (other terms) } \vspace{.2cm} \\
&=& \big( l_{i_1} \ldots l_{i_t} \big) \cdot \dfrac{\big(l_{i_1} \ldots l_{i_j} \big) \big( l_{i_{j+1}} \ldots l_{i_{t}} \big)}{\big( l_{i_1} \ldots l_{i_t} \big)} \cdot \text{ (other terms) }
\end{array}
\end{displaymath}
Hence we conclude that the greatest common divisor of zero trials which are not an edge cancels in the product $Q_G$ and so $Q_G = l_1 \cdots l_k$.
\end{proof}

Basis criteria for spline modules on trees over GCD domains can be given as follows.

\begin{theo} Let $G$ be a tree with $n$ vertices and $k$ edges. Then $\{ F_1 , \ldots , F_n \} \subset R_{(G , \alpha)}$ forms a basis for $R_{(G , \alpha)}$ if and only if $\big\vert F_1 \text{ } F_2 \text{ } \ldots \text{ } F_n \big\vert = r \cdot Q_G$ where $r \in R$ is a unit.
\label{ttheo}
\end{theo}

\begin{proof} The second part of the theorem can be proved exactly by the same method in the proof of Lemma~\ref{cycd331}. For the proof of the first part, we assume that  $\{ F_1 , \ldots , F_n \}$ forms a basis for $R_{(G , \alpha)}$. There exists a flow-up basis $\{ G_1 , \ldots , G_n \}$ for $R_{(G , \alpha)}$ since $(G , \alpha)$ is a tree. It can be seen easily that $\big\vert G_1 \text{ } G_2 \text{ } \ldots \text{ } G_n \big\vert = Q_G$. Hence $\big\vert F_1 \text{ } F_2 \text{ } \ldots \text{ } F_n \big\vert = r \cdot Q_G$ by Corollary~\ref{detcor3}, where $r \in R$ is a unit.
\end{proof}

We can also give basis criteria for graphs obtained by joining cycles, diamonds and trees together along common vertices. 

\begin{cor} Let $\{ G_1 , \ldots , G_k \}$ be a collection of cycles, diamond graphs and trees and let $G$ be a graph obtained by joining $G_1 , \ldots ,G_k$ together along common vertices which are cut vertices in $G$. Then $\{ F_1 , \ldots , F_n \} \in R_{(G,\alpha)}$ forms a basis for $R_{(G , \alpha)}$ if and only if $\big\vert F_1 \text{ } F_2 \text{ } \ldots \text{ } F_n \big\vert = r \cdot Q_{G_1} \cdots Q_{G_k}$ where $r \in R$ is a unit.
\end{cor}

\begin{proof} We give a sketch of the proof. First reorder the vertices of $G$ such as the vertices on each $G_i$ are consecutively ordered except the least indiced vertex. Then each basis element of $R_{(G_i, \alpha_i)}$ for all $i$ gives a basis element of $R_{(G,\alpha)}$. Construct the matrix whose columns are the elements of the obtained basis for $R_{(G,\alpha)}$. Notice that the determinant of this matrix is equal to the product $r \cdot Q_{G_1} \cdots Q_{G_k}$.
\end{proof}

If $R$ is a PID then the existence of flow-up bases is guaranteed (See~\cite{Alt}). Hence we can give basis criteria for spline modules on arbitrary graphs over principal ideal domains as follows:

\begin{theo} Let $(G , \alpha)$ be an edge labeled graph with $n$ vertices and $R$ be a PID. Then $\{ F_1 , \ldots , F_n \} \subset R_{(G,\alpha)}$ forms a module basis for $R_{(G,\alpha)}$ if and only if $\big\vert F_1 \text{ } F_2 \text{ } \ldots \text{ } F_n \big\vert = r \cdot Q_G$ where $r \in R$ is a unit.
\end{theo}

\begin{proof} Since $R$ is a PID, there exists a flow-up basis $\{ G_1 , \ldots , G_n \}$ for $R_{(G,\alpha)}$ such that $\big\vert G_1 \text{ } G_2 \text{ } \ldots \text{ } G_n \big\vert = Q_G$. Assume that $\{ F_1 , \ldots , F_n \} \subset R_{(G,\alpha)}$ forms a module basis for $R_{(G,\alpha)}$. Hence $\big\vert F_1 \text{ } F_2 \text{ } \ldots \text{ } F_n \big\vert = r \big\vert G_1 \text{ } G_2 \text{ } \ldots \text{ } G_n \big\vert = r \cdot Q_G$ by Corollary~\ref{detcor3}, where $r \in R$ is a unit. The other part of the theorem can be proved by similar techniques in the proof of Lemma~\ref{cycd331}.
\end{proof}

The element $Q_G$ depends on the graph type. The complexity of the formula is related to the number of the cycles contained in $G$ and the base ring $R$. We believe that general basis criteria for spline modules on arbitrary graphs over GCD domains can be given by $Q_G$. We claim the following conjecture:

\begin{conj} Let $(G , \alpha)$ be any edge labeled graph with $n$ vertices. Then $\{ F_1 , \ldots , F_n \} \in R_{(G,\alpha)}$ forms a module basis for $R_{(G,\alpha)}$ if and only if $\big\vert F_1 \text{ } F_2 \text{ } \ldots \text{ } F_n \big\vert = r \cdot Q_G$ where $r \in R$ is a unit.
\end{conj}

\section{Homogenization of Splines}

In this section we study splines over the polynomial ring $R = k[x_1 , \ldots , x_d]$. We first define homogeneous splines and the homogenization of a spline. Then we introduce the homogenization of an edge labeled graph. At the end of this section, we discuss freeness relations between the module of homogeneous splines and $R_{(G,\alpha)}$.

\begin{defn} Let $(G,\alpha)$ be an edge labeled graph over the polynomial ring $R = k[x_1 , \ldots , x_d]$ and let $F = (f_1 , \ldots , f_n) \in R_{(G,\alpha)}$. The degree of $F$ is defined as the total degree, which is
\begin{displaymath}
\text{deg } F = \text{max} \big\{ \text{deg } f_i \big\vert 1 \leq i \leq n \big\}.
\end{displaymath}
$F$ is called homogeneous if $\text{deg } f_i = \text{deg } F$ or $f_i = 0$ for all $i$.
\end{defn}

If $f \in R$, we define the homogenization ${}^h f \in \hat{R} = k[x_1 , \ldots , x_d , z]$ of $f$ by
\begin{displaymath}
{}^h f \big( x_1 , \ldots , x_d  , z\big) = z^{\partial f} f \left( \dfrac{x_1}{z} , \ldots , \dfrac{x_d}{z} \right)
\end{displaymath}
where $\partial f$ denotes the degree of $f$. If $F \in R^t$, the homogenization ${}^h F \in \hat{R^t}$ is defined by
\begin{displaymath}
{}^h F = {}^h \big( f_1 , \ldots , f_t \big) = \big( z^{\partial F - \partial f_1} ({}^h f_1) , \ldots , z^{\partial F - \partial f_t} ({}^h f_t) \big).
\end{displaymath}
where $\partial F$ denotes the maximum of the $\partial f_i$'s. If $f \in \hat{R}$, we set $f(1) = f \big( x_1 , \ldots , x_d , 1 \big)$ and if $F \in {\hat{R}}^t$ then $F(1) = \big( f_1 (1), \ldots , f_t (1) \big)$.

The homogenization operation satisfies the following properties:

\begin{prop} Let $F, G \in R^t$. Then
\begin{enumerate}[label=\emph{(\alph*)}]
	\item $z^{\partial F + \partial G} {}^h (F + G) = z^{\partial (F+G)} \big( z^{\partial G} {}^h F + z^{\partial F} {}^h G \big)$
	\item ${}^h F(1) = F$
\end{enumerate}
\label{homprop}
\end{prop}

\begin{proof} See Section 5 of Chapter VII in~\cite{Zar}.
\end{proof}

We know that $R_{(G,\alpha)}$ is an $R$-module. In general, $R_{(G,\alpha)}$ may not be a graded $R$-module with standard grading. In order to see this, consider the following example:

\begin{ex} Let $(G, \alpha)$ be as the figure below.
\begin{figure}[H]
\begin{center}
\scalebox{0.16}{\includegraphics{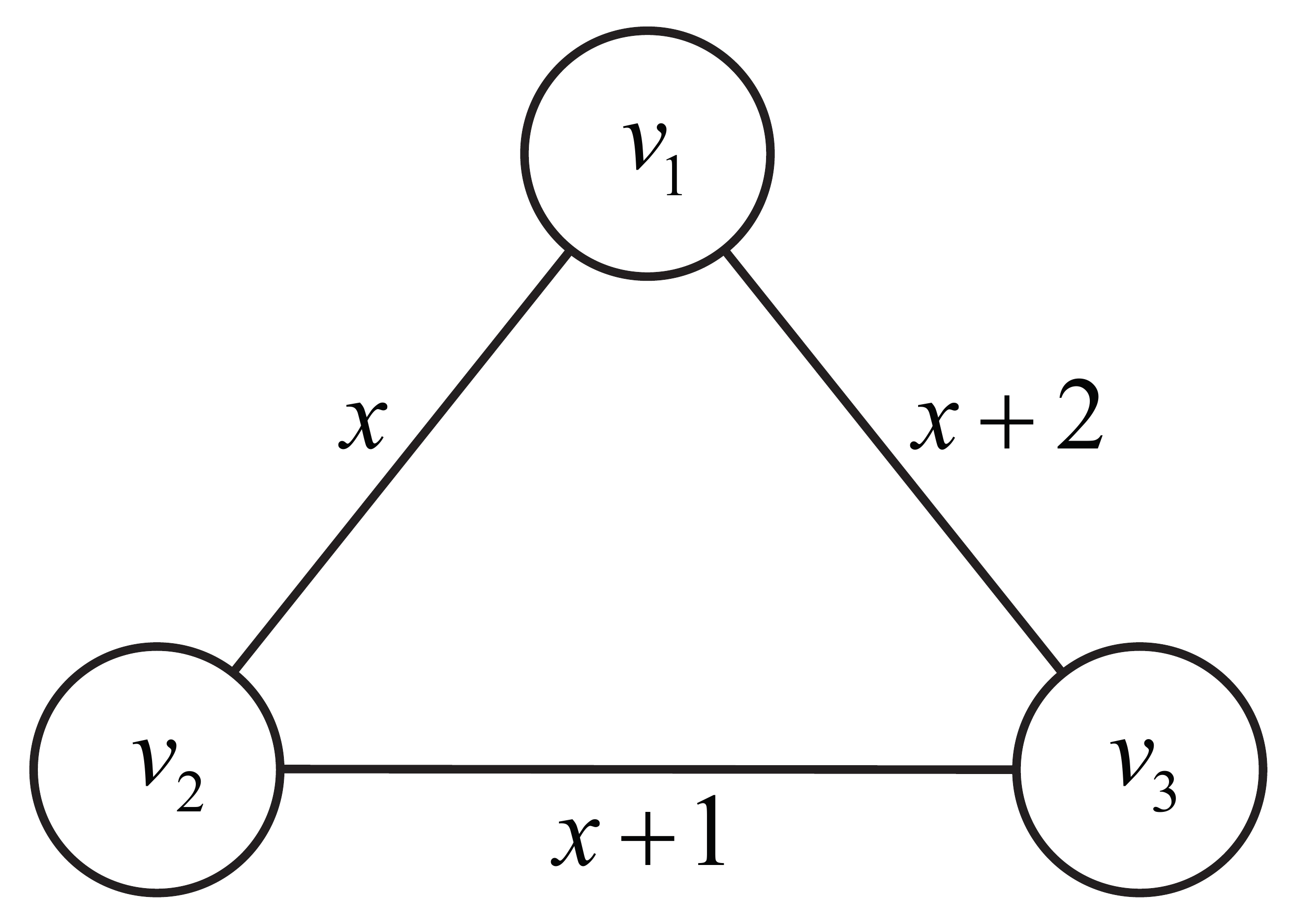}}
\caption{Example of spline module which is not graded}
\label{sp21}
\end{center}
\end{figure}
Take $F = \begin{bmatrix} x^2 + 2x + 1 \\ x+1 \\ 1 \end{bmatrix} \in R_{(G,\alpha)}$. If we rewrite $F$ as a sum of homogeneous splines by
\begin{displaymath}
F = \begin{bmatrix} x^2 + 2x + 1 \\ x+1 \\ 1 \end{bmatrix} = \begin{bmatrix} x^2 \\ 0 \\ 0 \end{bmatrix} + \begin{bmatrix} 2x \\ x \\ 0 \end{bmatrix} + \begin{bmatrix} 1 \\ 1 \\ 1 \end{bmatrix},
\end{displaymath}
then the first term of the sum, $(0,0,x^2)$ is not an element of $R_{(G,\alpha)}$. Hence $R_{(G,\alpha)}$ is not a graded $R$-module with the standard grading.

In order to obtain a graded module structure for $R_{(G,\alpha)}$, we define the homogenization of an edge labeled graph.
\end{ex}

\begin{defn} Let $(G,\alpha)$ be an edge labeled graph with base ring $R = k[x_1 , \ldots , x_d]$. The homogenization of $(G,\alpha)$ is defined by the edge labeling function $\hat{\alpha} : E \to \{ \text{ideals in } \hat{R} \}$ with $\hat{\alpha} (e) = {}^h \alpha(e)$ where $\hat{R} = k[x_1 , \ldots , x_d , z]$. Since the base ring is $\hat{R}$, we denote the set of splines on $(G,\hat{\alpha})$ by $\hat{R}_{(G,\hat{\alpha})}$.
\end{defn}

\begin{ex} The following figure illustrates the homogenization of an edge labeled graph:
\begin{figure}[H]
\begin{center}
\scalebox{0.16}{\includegraphics{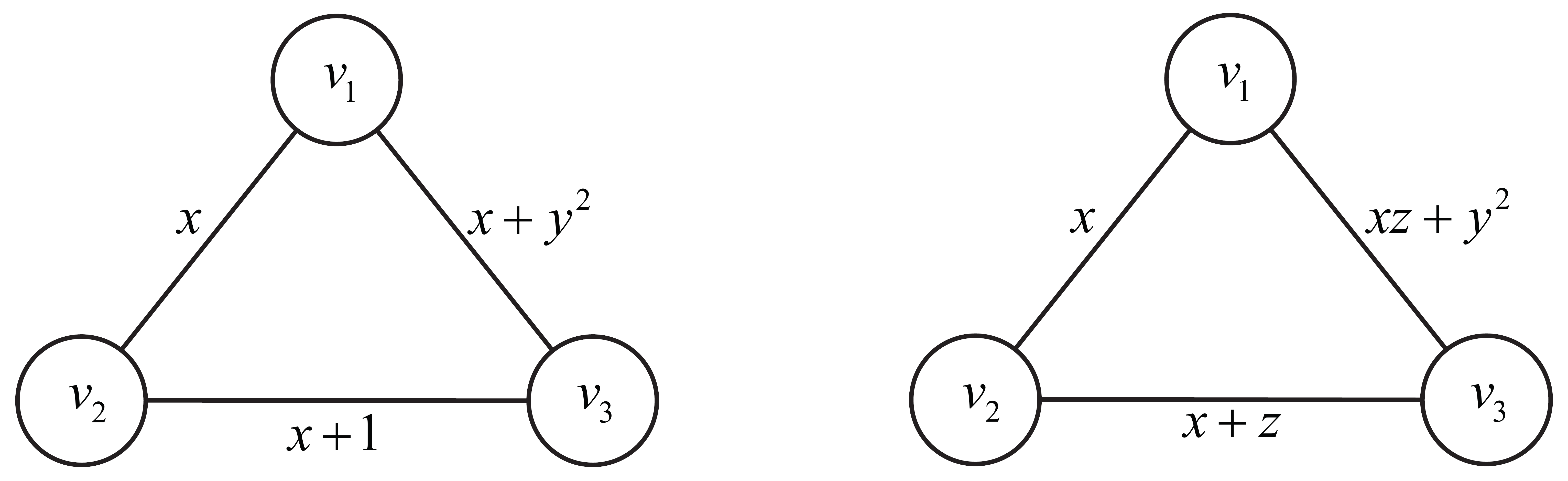}}
\caption{An edge labeled graph (left) and its homogenization (right)}
\label{sp2113}
\end{center}
\end{figure}
\end{ex}

\begin{lem} Let $F \in R_{(G,\alpha)}$. Then ${}^h F \in \hat{R}_{(G,\hat{\alpha})}$.
\label{homlem1}
\end{lem}

\begin{proof} Let $F = (f_1 , \ldots , f_t) \in R_{(G,\alpha)}$. Then
\begin{displaymath}
{}^h F = {}^h \big( f_1 , \ldots , f_t \big) = \big( z^{\partial F - \partial f_1} ({}^h f_1) , \ldots , z^{\partial F - \partial f_t} ({}^h f_t) \big).
\end{displaymath}
In order to see that ${}^h F \in \hat{R}_{(G,\hat{\alpha})}$, we need to check $z^{\partial F - \partial f_i} ({}^h f_i) - z^{\partial F - \partial f_j} ({}^h f_j) \in \hat{\alpha}(e_{ij})$ for all adjacent pair of vertices $v_i , v_j \in V(G)$. Let $v_i , v_j$ be two adjacent vertices of $G$. Since $F \in R_{(G,\alpha)}$, we have $f_i - f_j \in \alpha(e_{ij})$. Here we have
\begin{displaymath}
\begin{array}{ccl}
z^{\partial F - \partial f_i} ({}^h f_i) - z^{\partial F - \partial f_j} ({}^h f_j) &=& z^{\partial F - \partial f_i - \partial f_j} \big( z^{\partial f_j} ({}^h f_i) - z^{\partial f_i} ({}^h f_j) \big) \\ &=& z^{\partial F - \partial (f_i + f_j)} \underbrace{{}^h (f_i - f_j)}_{\in \hat{\alpha}(e_{ij})} \in \hat{\alpha}(e_{ij}).
\end{array}
\end{displaymath}
The last equality follows from Proposition~\ref{homprop} (a). Hence ${}^h F \in \hat{R}_{(G,\hat{\alpha})}$.
\end{proof}

\begin{lem} $\hat{R}_{(G,\hat{\alpha})}$ has a graded $\hat{R}$-module structure with standard grading.
\label{homlem2}
\end{lem}

\begin{proof} Let $F = (f_1 , \ldots , f_t) \in \hat{R}_{(G,\hat{\alpha})}$. Let $F_m = (f_{1m} , \ldots , f_{tm})$ be the homogeneous component of $F$ of degree $m$. In order to see that $\hat{R}_{(G,\hat{\alpha})}$ has a graded $\hat{R}$-module structure, it is sufficient to see that $F_m \in \hat{R}_{(G,\hat{\alpha})}$. Since $F \in \hat{R}_{(G,\hat{\alpha})}$, $f_i - f_j \in \hat{\alpha}(e_{ij})$ for all adjacent pair of vertices $v_i , v_j \in V(G)$ where $\hat{\alpha}(e_{ij})$ is a homogeneous ideal in $\hat{R}$. Hence all homogeneous components of $f_i - f_j$ belongs to $\hat{\alpha}(e_{ij})$ and so $F_m \in \hat{R}_{(G,\hat{\alpha})}$.
\end{proof}

\begin{rem} Let $R_{(G,\alpha)}$ be a free $R$-module with basis $\mathcal{G} = \{ G_1 , \ldots , G_t \}$. Then it is trivial to expect that $\hat{R}_{(G,\hat{\alpha})}$ is also free $\hat{R}$-module with basis ${}^h \mathcal{G} = \{ {}^h G_1 , \ldots , {}^h G_t \}$ but this is not true in general. Consider the following example:
\label{reducedrem}
\end{rem}

\begin{ex} Let $(G, \alpha)$ be as in the figure below.
\begin{figure}[H]
\begin{center}
\scalebox{0.16}{\includegraphics{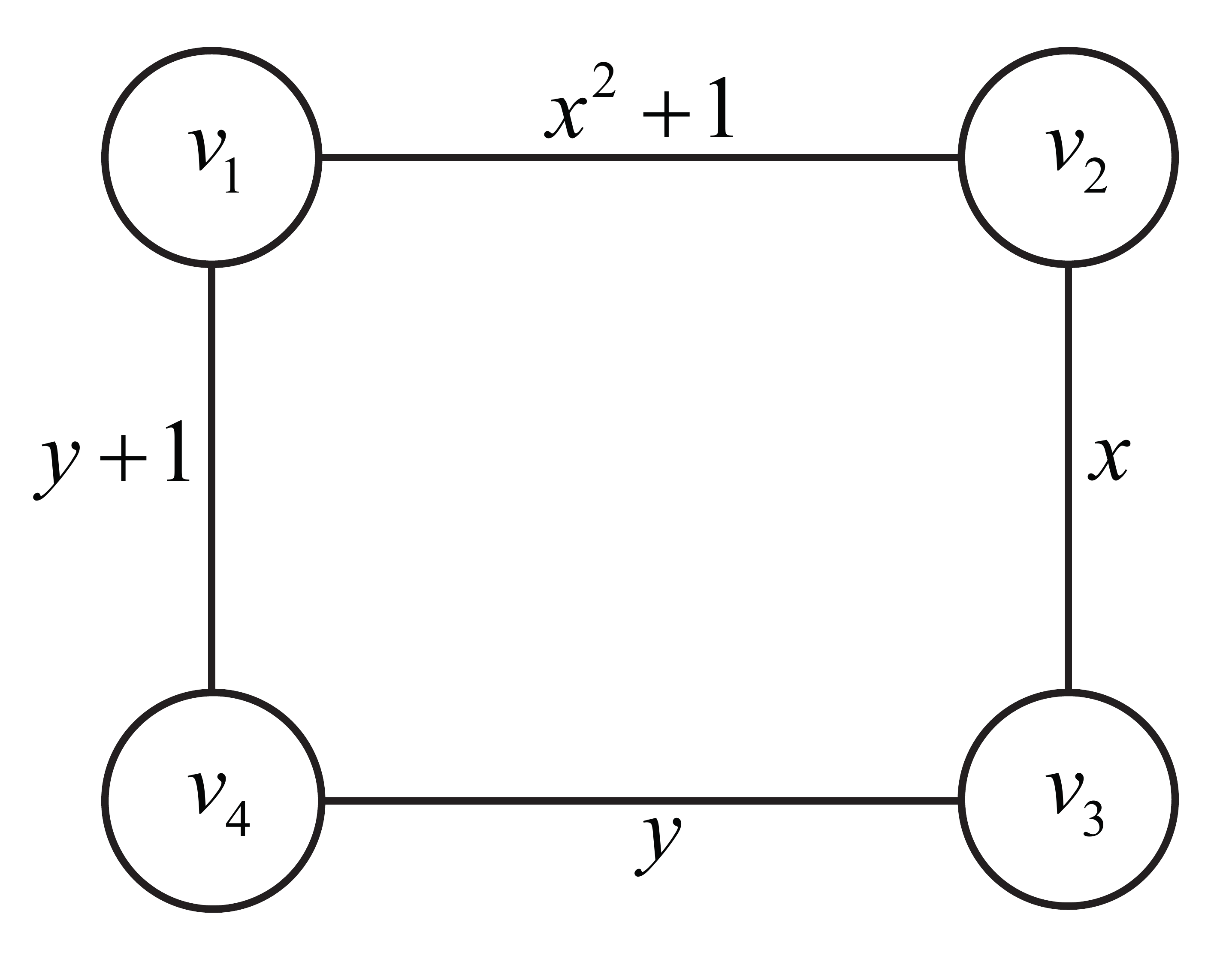}}
\caption{Edge labeled graph $(G, \alpha)$}
\label{sp71}
\end{center}
\end{figure}
\noindent
A flow-up basis for $R_{(G,\alpha)}$ is given by
\begin{displaymath}
\mathcal{G} = \left\{ \begin{bmatrix} 1 \\ 1 \\ 1 \\ 1 \end{bmatrix} , \begin{bmatrix} x^2 y + x^2 + y + 1 \\ x^2 + 1 \\ x^2 + 1 \\ 0 \end{bmatrix} , \begin{bmatrix} xy + x \\ x \\ 0 \\ 0 \end{bmatrix} , \begin{bmatrix} y^2 + y \\ 0 \\ 0 \\ 0 \end{bmatrix} \right\}.
\end{displaymath}
The homogenization of the edge labeled graph $(G, \alpha)$ is as in the figure below.
\begin{figure}[H]
\begin{center}
\scalebox{0.16}{\includegraphics{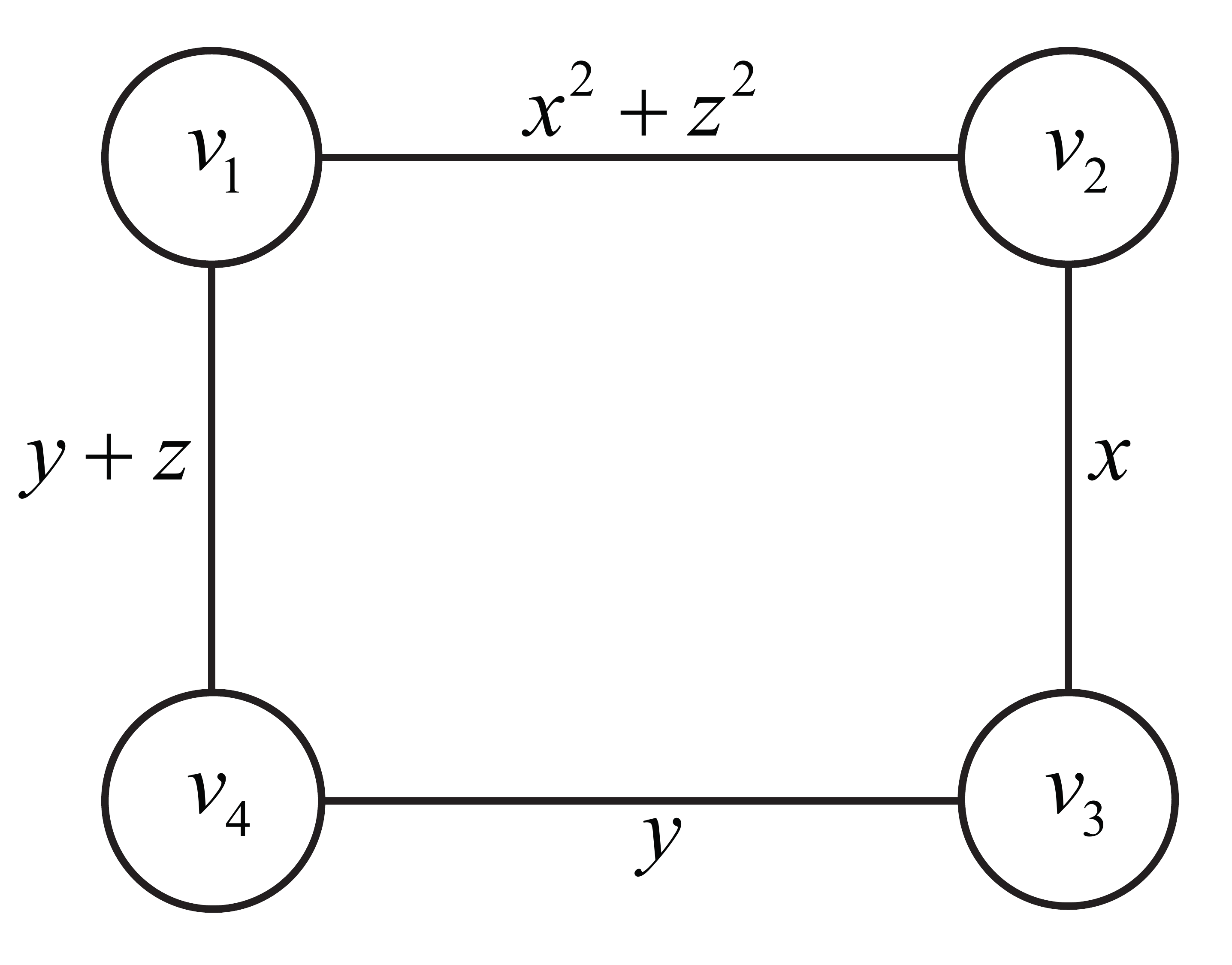}}
\caption{Homogenization of $(G, \alpha)$}
\label{sp72}
\end{center}
\end{figure}
\noindent
We obtain the set ${}^h \mathcal{G}$ by homogenizing the elements of $\mathcal{G}$.
\begin{displaymath}
{}^h \mathcal{G} = \left\{ \begin{bmatrix} 1 \\ 1 \\ 1 \\ 1 \end{bmatrix} , \begin{bmatrix} x^2 y + x^2 z + y z^2 + z^3 \\ x^2 z + z^3 \\ x^2 z + z^3 \\ 0 \end{bmatrix} , \begin{bmatrix} xy + xz \\ xz \\ 0 \\ 0 \end{bmatrix} , \begin{bmatrix} y^2 + yz \\ 0 \\ 0 \\ 0 \end{bmatrix} \right\}.
\end{displaymath}
Since ${}^h \mathcal{G}$ cannot generate $F = (0,0,xy,0) \in \hat{R}_{(G,\hat{\alpha})}$, it is not a basis for $\hat{R}_{(G,\hat{\alpha})}$. In fact, $\hat{R}_{(G,\hat{\alpha})}$ is not a free $\hat{R}$-module, although $R_{(G,\alpha)}$ is a free $R$-module.
\label{exreduced}
\end{ex}

A special type of basis for $R_{(G,\alpha)}$, which is called reduced basis, satisfies the expected property in Remark~\ref{reducedrem}.

\begin{defn} Let $(G,\alpha)$ be an edge labeled graph with $n$ vertices. Let $\mathcal{G} = \{ G_1 , \ldots , G_n \}$ be an $R$-module basis for $R_{(G,\alpha)}$. If for each $F \in R_{(G,\alpha)}$, $F = \sum r_i G_i$ implies $\text{deg } (r_i G_i) \leq \text{deg } F$, then $\mathcal{G}$ is called a reduced basis for $R_{(G,\alpha)}$.
\end{defn}

\begin{ex} The basis $\mathcal{G}$ of $R_{(G,\alpha)}$ in Example~\ref{exreduced} is not reduced. In order to see this, consider $F = (x,x,xy + x,x) \in R_{(G,\alpha)}$. Here we have

\begin{displaymath}
F = \underbrace{\begin{bmatrix} x \\ xy + x \\ x \\ x \end{bmatrix}}_{\text{degree } = 2} = \underbrace{x  \begin{bmatrix} 1 \\ 1 \\ 1 \\ 1 \end{bmatrix}}_{\text{degree } = 1} + \underbrace{0  \begin{bmatrix} x^2 y + x^2 + y + 1 \\ x^2 + 1 \\ x^2 + 1 \\ 0 \end{bmatrix}}_{\text{degree } = 0} + \underbrace{y  \begin{bmatrix} xy + x \\ x \\ 0 \\ 0 \end{bmatrix}}_{\text{degree } = 3} - \underbrace{x  \begin{bmatrix} y^2 + y \\ 0 \\ 0 \\ 0 \end{bmatrix}}_{\text{degree } = 3} .
\end{displaymath}
Some of the components on the right side has greater degree than $F$. Thus $\mathcal{G}$ is not a reduced basis for $R_{(G,\alpha)}$.
\label{exnotreduced}
\end{ex}

\begin{theo} \emph{\cite{Bil}} $R_{(G,\alpha)}$ has a reduced basis if and only if $\hat{R}_{(G,\hat{\alpha})}$ is free over $\hat{R}$.
\label{homtheo}
\end{theo}

\begin{proof} Proof is the same as in the classical case of splines which is given by Theorem 8.5. in~\cite{Bil}.
\end{proof}

Billera and Rose~\cite{Bil} gave a criteria that determines whether a basis of the module of classical splines is reduced or not. We give similar results for splines on cycles, diamond graphs and trees. First we need a lemma.

\begin{lem} Let $\{ F_1 , \ldots , F_n \} \subset R_{(G,\alpha)}$ be homogeneous elements. Then the degree of the determinant $\big\vert F_1 \text{ } F_2 \text{ } \ldots \text{ } F_n \big\vert$ is either $0$ or the sum of the degrees of the $F_i$'s.
\label{homlem3}
\end{lem}

\begin{proof} We use induction on $n$. If $n = 2$, then
\begin{displaymath}
\text{deg } \big\vert F \text{ } G \big\vert = \text{deg } \begin{vmatrix} f_1 & g_1 \\ f_2 & g_2 \end{vmatrix} = \text{deg } (g_1 f_2 - f_1 g_2).
\end{displaymath}
Here notice that each nonzero term of the products $f_1 g_2$ and $g_1 f_2$ has degree $\text{deg } F + \text{deg } G$, since $F$ and $G$ are homogeneous splines. Hence $\text{deg } \big\vert F \text{ } G \big\vert = \text{deg } (g_1 f_2 - f_1 g_2) = \text{deg } F + \text{deg } G$ or $0$.

Assume that for homogeneous splines $\{ F_1 , \ldots , F_k \} \subset R_{(G,\alpha)}$, we have $\text{deg } \big\vert F_1 \text{ } F_2 \text{ } \ldots \text{ } F_k \big\vert = \sum\limits_{i=1} ^k \text{deg } F_i$ or $0$. Let $\{ F_1 , \ldots , F_{k+1} \} \subset R_{(G,\alpha)}$ be homogeneous elements. By expanding the determinant\linebreak $\big\vert F_1 \text{ } F_2 \text{ } \ldots \text{ } F_{k+1} \big\vert$ on $(k+1)$-th column, which is $F_{k+1}$, we see that $\text{deg } \big\vert F_1 \text{ } F_2 \text{ } \ldots \text{ } F_{k+1} \big\vert = \sum\limits_{i=1} ^{k+1} \text{deg } F_i$ or $0$.
\end{proof}

\begin{theo} Let $(C_n , \alpha)$ be an edge labeled cycle with edge labels $\{ l_1 , \ldots , l_n \}$ and let $\mathscr{F} = \{ F_1 , \ldots , F_n \}$ be a basis for $R_{(C_n,\alpha)}$. If $\mathscr{F}$ is a reduced basis, then
\begin{displaymath}
\sum\limits_{i=1} ^n \emph{deg } F_i = \sum\limits_{i=1} ^n \emph{deg } l_i - \emph{deg } \big( l_1 , \ldots , l_n \big).
\end{displaymath}
\label{homtheo2}
\end{theo}

\begin{proof} Since $\mathscr{F}$ is a basis for $R_{(C_n,\alpha)}$, we have 
\begin{displaymath}
\big\vert F_1 \text{ } F_2 \text{ } \ldots \text{ } F_n \big\vert = \dfrac{l_1 \cdots l_n}{(l_1 , \ldots , l_n)}
\end{displaymath}
by Theorem~\ref{homtheo} and hence
\begin{displaymath}
\text{deg } \big\vert F_1 \text{ } F_2 \text{ } \ldots \text{ } F_n \big\vert = \sum\limits_{i=1} ^n \text{deg } l_i - \text{deg } \big( l_1 , \ldots , l_n \big).
\end{displaymath}
\noindent
Since $\mathscr{F}$ is reduced basis for $R_{(C_n,\alpha)}$, the set ${}^h \mathscr{F} = \{ {}^h F_1 , \ldots , {}^h F_n \}$ is a  basis for $\hat{R}_{(C_n,\hat{\alpha})}$ by Theorem~\ref{homtheo}. By using Theorem~\ref{cycthm} and Lemma~\ref{homlem3}, we have

\begin{displaymath}
\begin{array}{ccl}
\sum\limits_{i=1} ^n \text{deg } {}^h F_i &=& \text{deg } \big\vert {}^h F_1 \text{ } {}^h F_2 \text{ } \ldots \text{ } {}^h F_n \big\vert = \text{deg } \left( \dfrac{{}^h l_1 \cdots {}^h l_n}{({}^h l_1 , \ldots , {}^h l_n)} \right) \\ &=& \sum\limits_{i=1} ^n \text{deg } {}^h l_i - \text{deg } \big( {}^h l_1 , \ldots , {}^h l_n \big).
\end{array}
\end{displaymath}
Here notice that $\text{deg } F_i = \text{deg } {}^h F_i$ and $\text{deg } l_i = \text{deg } {}^h l_i$. Hence we get
\begin{displaymath}
\begin{array}{ccl}
\sum\limits_{i=1} ^n \text{deg } F_i &=& \sum\limits_{i=1} ^n \text{deg } {}^h F_i = \sum\limits_{i=1} ^n \text{deg } {}^h l_i - \text{deg } \big( {}^h l_1 , \ldots , {}^h l_n \big)\vspace{.3cm} \\ &=& \sum\limits_{i=1} ^n \text{deg } l_i - \text{deg } \big( l_1 , \ldots , l_n \big).
\end{array}
\end{displaymath}
\end{proof}

A similar statement of Theorem~\ref{homtheo2} can be given for diamond graph $D_{m,n}$ and trees as below. They can be proved similarly as Theorem~\ref{homtheo2}.

\begin{theo} Let $(D_{m,n} , \alpha)$ be the edge labeled diamond graph as in the Figure~\ref{dmn} and let $\mathscr{F} = \{ F_1 , \ldots , F_{m+n-2} \}$ be a basis for $R_{(D_{m,n},\alpha)}$. If $\mathscr{F}$ is a reduced basis, then
\begin{displaymath}
\sum\limits_{i=1} ^{m+n-2} \emph{deg } F_i = \sum\limits_{i=1} ^{m+n-1} \emph{deg } l_i - \emph{deg } \big( (l_2 , \ldots , l_n) (l_{n+1} , \ldots , l_{m+n-1}) \text{ , } l_1 (l_2 , \ldots , l_{m+n-1}) \big).
\end{displaymath}
\label{homtheo4}
\end{theo}

\begin{theo} Let $(G , \alpha)$ be an edge labeled tree with $n$ vertices and $k$ edges. Let $\mathscr{F} = \{ F_1 , \ldots , F_n \}$ be a basis for $R_{(G , \alpha)}$. If $\mathscr{F}$ is a reduced basis, then
\begin{displaymath}
\sum\limits_{i=1} ^n \emph{deg } F_i = \sum\limits_{i=1} ^k \emph{deg } l_i.
\end{displaymath}
\label{homtheo3}
\end{theo}

\end{document}